\documentclass[journal,twoside,web]{ieeecolor}
\usepackage{generic}
\usepackage{cite}
\usepackage{amsmath,amssymb,amsfonts}
\usepackage{algorithmic}
\usepackage{graphicx}
\usepackage{textcomp}

\usepackage[hidelinks]{hyperref}

\let\labelindent\relax
\usepackage{enumitem}
\usepackage{acro}

\DeclareAcronym{hccf}{
  short = HCCF,
  long = hyperbolic controller canonical form
}
\DeclareAcronym{hocf}{
  short = HOCF,
  long = hyperbolic observer canonical form
}
\DeclareAcronym{bcs}{
  short = BCS,
  long = boundary control system 
}
\DeclareAcronym{ndde}{
  short = NFDE,
  long = neutral functional differential equation
}
\DeclareAcronym{ac}{
  short = AC,
  long = analytic case
}
\DeclareAcronym{hc}{
  short = HC,
  long = hyperbolic case
}

\usepackage{amsmath}
\usepackage{amsfonts}
\usepackage{mathrsfs}

\renewcommand{\star}{\diamond}

\newcommand{\SysOp}{\Sigma}
\newcommand{\SysOpBcs}[1][]{\SysOp_{\text{bcs}#1}}
\newcommand{\SysOpBos}[1][]{\SysOp_{\text{bos}#1}}

\newcommand{\EA}{\mathcal{A}}
\newcommand{\EAcl}{\EA^{\text{cl}}}
\newcommand{\EB}{\mathcal{B}}
\newcommand{\EC}{\mathcal{C}}

\newcommand{\EK}{\mathcal{K}}
\newcommand{\EL}{\mathcal{L}}

\newcommand{\BA}{\mathfrak{A}}
\newcommand{\BB}{\mathfrak{B}}
\newcommand{\BC}{\mathfrak{C}}

\newcommand{\BK}{\mathfrak{K}}

\newcommand{\BR}{\mathfrak{R}}
\newcommand{\BG}{\varkappa}

\newcommand{\OA}{\mathscr{A}}

\newcommand{\OK}{\mathscr{K}}

\newcommand{\OR}{\mathscr{R}}

\newcommand{\MA}{A}

\newcommand{\MB}{B}

\newcommand{\MK}{K}

\newcommand{\mb}{b}
\newcommand{\mc}{c}

\newcommand{\OKc}{\breve{\OK}}
\newcommand{\BKc}{\breve{\BK}}
\newcommand{\Kc}{\breve{\EK}}

\newcommand{\BKu}{\mathring{\BK}}
\newcommand{\OKu}{\mathring{\OK}}
\newcommand{\Kb}{\EK}

\newcommand{\Lc}{\breve{\EL}}
\newcommand{\Lu}{\mathring{\EL}}
\newcommand{\Lb}{\EL}

\newcommand{\ctrlIsText}{\text{\rm{c}}}
\newcommand{\obsIsText}{\text{\rm{o}}}
\newcommand{\desSText}{\text{\rm{d}}}
\newcommand{\desCsText}{\text{\rm{dc}}}
\newcommand{\desOsText}{\text{\rm{do}}}
\newcommand{\clSText}{\text{\rm{cl}}}

\newcommand{\ctrlIs}[2][]{#2^{\ctrlIsText #1}}
\newcommand{\obsIs}[2][]{#2^{\obsIsText #1}}
\newcommand{\desS}[2][]{#2^{\desSText #1}}
\newcommand{\desCs}[2][]{#2^{\desCsText #1}}
\newcommand{\desOs}[2][]{#2^{\desOsText #1}}
\newcommand{\clS}[2][]{#2^{\clSText #1}}

\newcommand{\Ss}{\mathcal{X}}
\newcommand{\Is}{\mathcal{U}}
\newcommand{\Os}{\mathcal{Y}}

\newcommand{\Nats}{\mathbb{N}}
\newcommand{\Integers}{\mathbb{Z}}
\newcommand{\Reels}{\mathbb{R}}
\newcommand{\Compl}{\mathbb{C}}

\newcommand{\LOp}{\mathscr{L}}
\newcommand{\Lz}{L^{2}}
\newcommand{\lz}{l^{2}}
\newcommand{\li}{l^{\infty}}

\newcommand{\Hn}[1]{H^{#1}}

\newcommand{\eval}{\lambda}
\newcommand{\evec}{\varphi}

\newcommand{\svec}{\psi}

\newcommand{\dual}[3][]{\langle #2 , #3 \rangle_{D(\EA^*_{#1})}}
\newcommand{\dualf}[3]{\langle #1 , #2 \rangle_{D(#3)}}
\newcommand{\dualc}[2]{\langle #1 , #2 \rangle_{D(\ctrlIs[*]{\EA})}}
\newcommand{\dualo}[3][]{\langle #2 , #3 \rangle_{D(\obsIs[*]{\EA}_{#1})}}
\newcommand{\dualdc}[2]{\langle #1 , #2 \rangle_{D(\desCs[*]{\EA})}}

\newcommand{\scal}[2]{\langle #1 , #2 \rangle_{\Ss}}
\newcommand{\scalf}[3]{\langle #1 , #2 \rangle_{#3}}

\newcommand{\scale}[2]{\langle #1 , #2 \rangle_{\Ss_\scc}}

\newcommand{\s}{x}
\newcommand{\ms}{p}
\newcommand{\uu}{u}

\newcommand{\yy}{y}

\newcommand{\sh}{\hat{\s}}
\newcommand{\yh}{\hat{\yy}}
\newcommand{\st}{\tilde{\s}}
\newcommand{\yt}{\tilde{\yy}}

\newcommand{\sw}{w}
\newcommand{\ca}{\alpha}
\newcommand{\cb}{\beta}
\newcommand{\cc}{\gamma}

\newcommand{\kappac}{\kappa_{\text{c}}}
\newcommand{\kappao}{\kappa_{\text{o}}}
\newcommand{\muc}{\mu_{\text{c}}}
\newcommand{\muo}{\mu_{\text{o}}}

\newcommand{\scc}{\eta}
\newcommand{\soc}{\xi}
\newcommand{\sfc}{\nu}
\newcommand{\fo}{\chi}
\newcommand{\foa}{\zeta}
\newcommand{\tho}{\theta_{+}}
\newcommand{\thu}{\theta_{-}}
\newcommand{\acc}{a}
\newcommand{\accb}{\tilde{a}}
\newcommand{\accu}{\mathring{a}}

\newcommand{\Trafo}{T}

\newcommand{\diff}[1]{\partial_{#1}}

\newcommand{\CN}{$\text{C}_{0}$}

\renewcommand{\Re}{\text{Re}}

\newcommand{\SG}{\mathcal{T}}

\newcommand{\djimin}{d_{j,i}^{\text{min}}}
\newcommand{\djimax}{d_{j,i}^{\text{max}}}

\newcommand{\Transpose}{\mathsf{T}}

\newcommand{\cconj}[1]{\overline{#1}}

\newcommand{\Dirac}[1]{\delta_{#1}}

\usepackage{accents}
\newcommand{\uubar}[1]{\underaccent{\bar}{#1}}
\newcommand{\adverseComp}[1]{\uubar{#1}}

\newtheorem{thm}{Theorem}[section]
\newtheorem{lem}[thm]{Lemma}

\newtheorem{rem}[thm]{Remark}
\newtheorem{defn}[thm]{Definition}
\newtheorem{assum}[thm]{Assumption}
\usepackage{todonotes}

\def\BibTeX{{\rm B\kern-.05em{\sc i\kern-.025em b}\kern-.08em
    T\kern-.1667em\lower.7ex\hbox{E}\kern-.125emX}}
\begin{document}
\title{Late lumping of transformation-based feedback laws for boundary control systems}

\author{Marcus Riesmeier, Frank Woittennek
\thanks{This work has been submitted to the IEEE for possible publication. Copyright may be transferred without notice, after which this version may no longer be accessible.}
\thanks{Marcus Riesmeier and Frank Woittennek are with UMIT Tirol, Private University for Health Sciences and Technology, Eduard-Wallnöfer-Zentrum 1, 6060 Hall in Tirol, Austria (e-mail: \{marcus.riesmeier, frank.woittennek\}@umit-tirol.at).}}

\maketitle
\begin{abstract}
Late-lumping feedback design for infinite-dimensional linear systems with unbounded input operators is considered. The proposed scheme is suitable for the approximation of backstepping and flatness-based designs and relies on a decomposition of the feedback into a bounded and an unbounded part. Approximation applies to the bounded part only, while the unbounded part is assumed to allow for an exact realization. Based on spectral results, the convergence of the closed-loop dynamics to the desired dynamics is established. By duality, similar results apply to the approximation of the observer output-injection gains for systems with boundary observation. The proposed design and approximation steps are demonstrated and illustrated based on a hyperbolic infinite-dimensional system.
\end{abstract}


\section{Introduction}
\label{sec:intro}
For the control and observer synthesis for infinite-dimensional systems, there exist different
paradigms. Among these, so-called early-lumping designs, see for example~\cite{Harkort2014}, are certainly the 
most popular in practice. Thereby, the controller and the observer are designed for a finite-dimensional approximation of the infinite-dimensional system.
Despite the numerous advantages, in particular the large number of  design methods applicable to finite-dimensional systems,
a main drawback of early-lumping methods is that stability of
the infinite-dimensional closed-loop system is not automatically guaranteed.
Direct methods provide an efficient alternative for the design of simple stabilizing
control laws. Within these methods, control will be designed directly on the basis of the infinite-dimensional system description. Typical examples of such designs are
collocated feedback laws obtained  as a result of energy-based design schemes, see for example~\cite{Villegas2007Phd}. Despite their undoubted elegance and simplicity,
the achievable closed-loop dynamics are limited.

Within the present contribution a third approach is pursued, the so-called late-lumping design, which came up implicitly along with the development of flatness-based \cite{WoittennekRudolph2012mathmod} and, even more, backstepping-based ~\cite{Krstic2008bcp} designs. Similarly as for direct methods, the controller and the observer will be designed for the original infinite-dimensional description of the plant.
These techniques aim in assigning desired closed-loop dynamics
to the system under consideration. This is achieved by viewing the system in
particular coordinates, which allows for simple control design, similar to the
canonical forms, well known from finite-dimensional linear systems
theory. Therefore, within this contribution, the described design techniques are referred to as transformation-based designs or as design by dynamics assignment.
Similar techniques also apply to the observer design.
Although the described techniques allow for a flexible assignment of desired closed-loop dynamics as they rely on the feedback of the infinite-dimensional state, for the same reason, they require subsequent approximation of the infinite-dimensional controller and observer schemes. This motivates the term ``late-lumping'' design.
Similarly as for the early-lumping approach, late-lumping may lead to stability issues in the closed-loop dynamics. On the one hand, the final control scheme is an arbitrarily accurate approximation of a feedback designed on the basis of the infinite-dimensional system description. It is, therefore, reasonable to expect that the obtained closed-loop dynamics are close to the desired dynamics that would have been achieved with the infinite-dimensional control law.
On the other hand, most of the systems considered within the late-lumping  approach are
boundary-controlled resp.\ posses a boundary observation. In the usual abstract state-space setting, this leads to unbounded control and observation operators. Therefore, the verification of the above expectations is not immediate. However, with a few exceptions, these problems have not been explicitly considered. 
In \cite{Auriol2019aut} stability of the closed-loop system with the approximated feedback law has been addressed for particular examples using Lyapunov techniques.
Previous results in this direction come e.g. from~\cite{Rebarber1989ieee}, where a spillover result is given for a beam equation subject to a modal approximated control law. 
Another result~\cite{GrueneMeurer2022}, showing closed-loop stability, for finite-dimensional observer-based state feedback, was given for abstract systems with discrete real-valued spectrum, by using the small-gain theorem.

In \cite{Woittennek2017ifac} only convergence of the feedback itself has been addressed without considering the closed-loop dynamics. Moreover, in \cite{Riesmeier2018cdc,Riesmeier2021pamm}
the authors propose, a modal approximation technique which superfluous the exact determination of the underlying state-transform. This essentially simplifies the implementation of the designs.
The above results and questions still open constitute the main motivation for the present article.

Within this contribution, infinite-dimensional systems with boundary control
resp.\ boundary observation are considered. Further assumptions to the system
class are formulated for the desired closed-loop system, instead of the original
control system.  In particular, the desired closed-loop operator is assumed to
be a discrete, Riesz-spectral, and possesses only simple eigenvalues. These
assumptions apply to both controller design and observer design. Further
assumptions are related to the spectral expansion of the unbounded input
or observation operators. These assumptions are necessary to allow the application of the results from~\cite{XuSallet1996siam}.
The article addresses both the controller and the observer design.
For the controller design no measurement resp.\ no observer is considered
while for the observer design no input resp.\ controller is considered. As a consequence, the approximation of observer-based output feedback is not addressed within this contribution and is left open for further research. For both scenarios, approximation schemes will be provided 
ensuring the convergence of the closed-loop spectrum to
the spectrum of the desired closed-loop system.
Concerning the application of the derived results to particular plants, two classes of systems are discussed, often occurring in physical and technical applications.
These considerations are further detailed for hyperbolic systems. 

Although using completely different techniques, the present contribution can be
seen as an extension of the results provided in \cite{Auriol2019aut} several
directions. Firstly, as already stated in \cite{Woittennek2017ifac}, the provided results are not restricted to a particular design method, e.g., backstepping design or flatness-based design.
Secondly, the class of systems considered is rather generic, i.e., only restricted by the properties of the chosen closed-loop dynamics. Finally, the convergence results are not restricted to
some stability margin of the closed-loop system but apply to the convergence of the complete spectrum.

The article is organized as follows. In Section~\ref{sec:prelim} the system
class will be introduced in detail and some theoretical background will be
recalled.  Section~\ref{sec:approx_fb} recalls the controller approximation
scheme and provides the spectral convergence result.  By duality, these results
are aligned to the approximation scheme for the observer gain, in
Section~\ref{sec:approx_obs}.  In Section~\ref{sec:ana} and
Section~\ref{sec:hyp} the application of the results will be discussed for
analytic and hyperbolic systems, respectively.  Section~\ref{sec:conclusion}
summarizes the article.

\section{Preliminaries}
\label{sec:prelim}
Within this section the notation and the structural
properties of the systems and designs under
consideration will be introduced.

\subsection{Basic notation}
As usual, $\Nats$, $\Reels$, $\Reels^+,\Compl$  denote the sets of positive integers, real numbers, non-negative real numbers, and complex numbers, respectively.
The complex conjugate of a complex number $c\in\Compl$ is denoted by
 $\cconj{c}$.
For given $n\in\Nats$, $\Compl^n$ is the
usual $n$-dimensional vector space of complex valued $n$-tuples over $\Compl$.
An element from $\Compl^{n\times m}$ is a complex valued $n\times m$ matrix.

Moreover, $\Lz(a,b;\Compl^n)$ denotes
the Lebesgue space of square-integrable functions $f:[a,b]\to\Compl^n$, $z\mapsto f(z)$,
while $\Hn{n}(a,b;\Compl^n)$ is the usual Sobolev space of $n$ times weakly  differentiable (in $\Lz(a,b;\Compl^n)$) functions on $[a,b]$ taking values in $\Compl^n$. 

The partial derivative of order $n\in\Nats$ w.r.t.\ a variable $z$ is denoted by $\diff{z}^n$.
Throughout this paper, $t\in\Reels$ stands exclusively for the time variable, the first  (partial) derivative w.r.t.\ $t$ of a function $h$ is abbreviated by $\dot{h}$.
For  two Banach spaces  $\mathcal{M}$ and $\mathcal{N}$, $\LOp(\mathcal{M}, \mathcal{N})$ denotes the
Banach space of linear bounded operators $\mathcal{M}\to\mathcal{N}$.

Let $\Ss$ denote a separable Hilbert space and
$\EA:\Ss\rightarrow \Ss$ a linear operator,
which is not necessarily bounded on $\Ss$.
The spectrum and the point spectrum of $\EA$ are denoted by $\sigma(\EA)$ and $\sigma_p(\EA)$, respectively.
Furthermore, $\{\eval_i\}$ denotes the
sequence of eigenvalues of $\EA$ and $\{\evec_i\}$ the corresponding sequence of
eigenvectors. The adjoint operator of $\EA$ is denoted by~$\EA^*$, with
eigenvalues~$\{\eval_i^*\}$ and eigenvectors $\{\evec_i^*\}$.
Moreover, $D(\EA)$ is the domain of $\EA$ and
$D(\EA)'$ is the dual space of $D(\EA)$. These spaces 
 are equipped with the graph norm and
the corresponding dual norm, respectively.
The duality pairing in $D(\EA^*)$ is denoted by
$\dual{F}{g},\,F\in D(\EA^*)',\,g\in D(\EA^*)$ and
the scalar product in $\Ss$ is denoted by
$\scalf{f}{g}{\Ss},\,f,g\in\Ss$. The scalar product as well as the duality
pairing take complex conjugation on the second argument.
The space of square-summable sequences and the space of bounded
sequences are denoted by $\lz$ and $\li$, respectively.
Finally, $\Dirac{r}(\cdot)$ is the Dirac delta distribution centered in $r\in\Reels$.

\subsection{System structure}
\label{sec:prelim:sys_struct}
First, the structure of a boundary control system~\cite{Fattorini1968siam} will be recalled,
then the dual property of boundary observation will be characterized in terms of the adjoint system.
For both, the common abstract state space representations will be given.

\subsubsection{Systems with boundary control}
Boundary control systems are of the form\footnote{Note that any system given in the seemingly more general form $\dot{\s}(t) = {\BA} \s(t) + \BB\uu(t),\,\BB\in\Ss$, \eqref{eq:bc_system_input} can be restated as $\dot{\s}(t) = ({\BA} + \BB\BR)\s(t),\,\BB\in\Ss$, \eqref{eq:bc_system_input} and is, therefore, covered by \eqref{eq:bc_system}.}
\begin{subequations}
\label{eq:bc_system}
\begin{align}
\label{eq:bc_system_dynamics}
\dot{\s}(t) &= \BA \s(t), && \s(0) = \s_0 \in \Ss \\
\label{eq:bc_system_input}
\uu(t) &= \BR \s(t)
\end{align}
\end{subequations}
with state $\s(t)\in\Ss$ and input\footnote{
  The considerations of the Subsections~\ref{sec:prelim:sys_struct}, \ref{sec:late_lumping} and \ref{sec:trafo_based}
  also apply to multi-input systems.
  However, since the convergence results given in Section~\ref{sec:approx_fb}
  are derived for single-input systems only, the system class is restricted from the beginning.
} $\uu(t)\in\Is=\Compl$,
cf.~\cite{Fattorini1968siam}.
The state space $\Ss$ is a separable Hilbert space and
$\BA:\Ss\supset D(\BA) \rightarrow \Ss$ and
$\BR:\Ss\supset D(\BA) \rightarrow \Is$
are unbounded operators on $\Ss$.

It is convenient to consider~\eqref{eq:bc_system} also in the form
\begin{align}
  \label{eq:sigma_system}
\dot{\s}(t) &= \SysOpBcs \left(\s(t),\uu(t)\right), && \s(0) = \s_0 \in \Ss,
\end{align}
where
$\SysOpBcs:\Ss\times\Is \supset D(\SysOpBcs) \rightarrow \Ss$
is unbounded
and $D(\SysOpBcs)=\{ (h_\s,h_\uu)\in D(\BA)\times\Is \,|\, \BR h_{\s}=h_{\uu} \}$
is dense in $\Ss\times\Is$.

For a unified treatment of the controller and observer design, a reformulation of
\eqref{eq:sigma_system} (resp.\ \eqref{eq:bc_system}) and \eqref{eq:dual_sys} as evolution equations are considered. More precisely,
\eqref{eq:sigma_system} is associated with
\begin{subequations}
\label{eq:system}
\begin{align}
\label{eq:system_dynamics}
\dot{\s}(t) &= \EA \s(t) + \EB \uu(t), && \s(0) = \s_0 \in \Ss
\end{align}
\end{subequations}
as described in~\cite[Chapter 3]{BensoussanPratoDelfourMitter2007}.
Therein, the system operator $\EA:\Ss\supset D(\EA) \rightarrow \Ss$
and the input operator $\EB\in\LOp(\Is, D(\EA^*)')$
are defined by the following relations: 
\begin{subequations}
  \label{eq:bc_ev_relation}
\begin{align}
&D(\EA) = \{h \in D(\BA) \,| \, \BR h = 0 \} && \\
&\EA\, h = \BA\, h, && \forall\,h\in D(\EA) \\
&\dualf{\EB}{h}{\SysOpBcs^{*}}=\scalf{\left(0, 1\right)}{\SysOpBcs^*h}{\Ss\times\Is}, && \forall\,h\in D(\SysOpBcs^*).
\end{align}
\end{subequations}
Throughout this contribution, $\EA$ is assumed to be the infinitesimal generator 
of a \CN-semigroup on $\Ss$, while $\EB$ is not required to be admissible\footnote{
Instead of admissibility of the input operator admissibility of
the feedback operator is required to be admissible within this contribution, cf.~\cite{Rebarber1989ieee}.}
in the sense of~\cite{TucsnakWeiss2009}.

\subsubsection{Systems with boundary observation}
Consider the system%
\begin{subequations}
\label{eq:bo_system}
\begin{align}
\label{eq:bo_system_dynamics}
\dot{\s}(t) &= \EA \s(t), && \s(0) = \s_0 \in \Ss \\
\label{eq:bo_system_bc}
\yy(t) &= \EC \s(t)
\end{align}
\end{subequations}
with state $\s(t)\in\Ss$ and output $\yy(t)\in\Os=\Compl$,
where $\EC:\Ss\supset D(\EA) \rightarrow \Os$ is the
unbounded observation operator. 
This system can also be written in the form
\begin{align}
  \label{eq:bo_sigma_system}
  (\dot{\s}(t),\yy(t)) = \SysOpBos \s(t), && \s(0)=\s_0\in\Ss,
\end{align}
where $\SysOpBos:\Ss \supset D(\SysOpBos) \rightarrow \Ss\times\Os$
is unbounded and $D(\SysOpBos)=D(\EA)$.

  System~\eqref{eq:bo_system} (resp.\ \eqref{eq:bo_sigma_system}) is called a system with boundary observation,
  if the adjoint system
  \begin{align}
    \label{eq:dual_sys}
  \dot\s^*(t) = \SysOpBos^*\left(\s^*(t), \yy^*(t)\right)
  \end{align}
  with
  input $\yy^*(t)\in\Os$ is a boundary control system, i.e.,
there exist operators
\begin{align*}
\OA:D(\OA)\to \Ss, && \OR:D(\OA)\to \Os
\end{align*}
with similar properties
as  $\BA$, $\BR$ such that
\begin{multline*}
\SysOpBos^*(\s^*(t),\yy^*(t))=\OA \s^*(t), \\
    D(\SysOpBos^*)=\{(h_{\s^*},h_{\yy^*})\in D(\OA)\times\Os \,|\, \OR h_{\s^*}=h_{\yy^*} \}.
\end{multline*}

The state space representation of the adjoint system \eqref{eq:dual_sys} is given by
\begin{subequations}
\label{eq:system_dual}
\begin{align}
\label{eq:system_dual_dynamics}
\dot{\s}^*(t) &= \EA^* \s^*(t) + \EC^* \yy^*(t), && \s^*(0) = \s^*_0 \in \Ss,
\end{align}
\end{subequations}
where
\begin{align*}
&D(\EA^*) = \{h \in D(\OA) \, | \, \OR h = 0 \} &&\\
&\EA^*h= \OA h,&&\forall\,h\in D(\EA^*)\\
&\dualf{\EC^*}{h}{\SysOpBos}=\scalf{\left(0, 1\right)}{\SysOpBos h}{\Ss\times\Os}, &&\forall\,h\in D(\SysOpBos).
\end{align*}

\subsection{Design by dynamics assignment}
\label{sec:late_lumping}
A common feature of the designs considered within this contribution is the idea not only to design
stabilizing feedback resp.\ convergent observers but, explicitly prescribe a 
desired closed loop dynamics. As outlined within the introduction, two
particular cases are treated within this contribution, the controller design by
state feedback and the
observer design for the autonomous system. The configurations considered are briefly introduced below. Note that the detailed analysis of the corresponding approximation
schemes will be described within Sections~\ref{sec:approx_fb} and \ref{sec:approx_obs}.

\subsubsection{State feedback design}
\label{sec:late_lumping_coc}
Starting from the boundary control system~\eqref{eq:bc_system}, the above-introduced idea consists in replacing the
boundary condition \eqref{eq:bc_system_input} by 
\begin{equation}\label{eq:bc_system_input:desired}
  \desCs{\uu}(t)=\desCs{\BR}\s(t),\quad \desCs{\BR}\in\LOp(D(\BA),\Is)
\end{equation}
with new input $\desCs{\uu}$ and the desired boundary operator $\desCs{\BR}$.
In a state-space setting, the desired closed loop system is given by
\begin{align}
\label{eq:closed_loop_system}
  \dot{\s}(t)&=\desCs{\EA} \s(t)+\desCs{\EB}\desCs{\uu}(t)
\end{align}
where the operators $\desCs{\EA}$ and $\desCs{\EB}$
are deduced from \eqref{eq:bc_system}, with \eqref{eq:bc_system_input} replaced by \eqref{eq:bc_system_input:desired}, in the same way as $\EA$ and $\EB$ in the previous section.

It remains to compute the feedback gain achieving the desired closed loop system.
Combining \eqref{eq:bc_system_input} and \eqref{eq:bc_system_input:desired}, it is obvious, that, starting from \eqref{eq:bc_system} any  feedback of the form
\begin{align}
  \label{eq:bcs_frak_fb}
  \uu(t) = \BKc \s(t)+\BG\,\desCs{\uu}(t),
\end{align}
with feedback gain
\begin{align}
  \label{eq:bcs_frak_gain}
  \BKc = \BR - \BG\,\desCs{\BR}\in\LOp(D(\BA),\Is),
\end{align}
where $\BG$ is an arbitrary nonzero real constant,
yields the desired closed-loop dynamics.
Although the particular representation of the feedback depends on $\BG$, all these representations are equivalent.
The constant $\BG$ should be chosen such that $\BKc$ takes a convenient form, for
the purpose of implementation.

Up to now a link between the original system  and the corresponding closed-loop system, both given as  boundary control systems, has been established via the feedback \eqref{eq:bcs_frak_fb}. Moreover, the corresponding state-space descriptions \eqref{eq:system} and \eqref{eq:closed_loop_system} have been deduced from the descriptions as boundary control systems. It remains to establish a direct link between these state-space descriptions, i.e., to deduce from \eqref{eq:bcs_frak_gain} a  feedback
\begin{align}
  \label{eq:bcs_cal_fb}
\uu(t) = \Kc \s(t) + \BG\,\desCs{\uu}(t), && \Kc\in\LOp(D(\desCs{\EA}),\Is),
\end{align}
such that $\desCs{\EA}=\EA+\EB\Kc$.
This is achieved by restricting the domain of $\BKc$ to $D(\desCs{\EA})$:
  \begin{align}
    \label{eq:bcs_cal_fb_gain}
  \Kc\s(t) = \BKc\s(t) = \BR\s(t), && \s(t)\in D(\desCs{\EA}).
  \end{align}
\begin{rem}
  \label{rem:adc_meaning}
The expression $\desCs{\EA}=\EA+\EB\Kc$
has a formal meaning within this contribution, since it is not
immediately clear how to read the operator $\EA+\EB\Kc$,
the precise definition of $\desCs{\EA}$ was given at the top of this subsection,
in terms of the boundary control system.
\end{rem}

\subsubsection{State observer design}
\label{sec:late_lumping_ooc}
For the observer design,
system~\eqref{eq:bo_system} together
with the observer
\begin{subequations}
\label{eq:intro_obs_only}
\begin{align}
\dot{\sh}(t) &=  \EA\sh(t) + \Lc\yt(t), && \yt(t) = \yh(t) - \yy(t) \\
\yh(t) &= {\EC}\sh(t),
\end{align}
\end{subequations}
is considered.
The observer gain $\Lc\in\LOp(\Os,D(\desOs[*]\EA)')$ has to be designed such
that the observer error system\footnote{
Like the $\desCs{\EA}=\EA+\EB\Kc$ from the previous subsection also 
$\desOs{\EA}=\EA + \Lc\EC$ has a formal meaning within this contribution
since the precise definition of $\desOs{\EA}$ is given in terms of the adjoint
system under boundary control, cf.\ Remark~\ref{rem:adc_meaning}.}
\begin{align*}
\dot{\st}(t) = (\EA + \Lc{\EC}) \st(t)=\desOs{\EA}\st(t), && \st(t)=\sh(t)-\s(t)
\end{align*}
possesses the desired dynamics $\desOs{\EA}: \Ss\supset D(\desOs{\EA}) \rightarrow \Ss$.
The operator $\desOs{\EA}$, in particular its domain, is defined as the adjoint of $\desOs[*]{\EA}$ with
\begin{align*}
D(\desOs[*]{\EA}) = \{h\in D(\OA) | \OR h = \OKc h\},
\end{align*}
where $\OKc\in\LOp(D(\OA),\Os)$, corresponding to $\BKc$, is the feedback operator
of the adjoint system~\eqref{eq:dual_sys} with dynamics $\desOs[*]{\EA}$.
This means, that system~\eqref{eq:dual_sys} under the feedback
\begin{align*}
  \yy^*(t) = \OKc\s^*(t) + \BG\,\desOs[*]{\yy}(t)
\end{align*}
with feedback gain
\begin{align*}
  \OKc = \OR - \BG\,\desOs{\OR}\in\LOp(D(\OA),\Os),
\end{align*}
has the state-space representation
\begin{align*}
  \dot{\s}^*(t)&=\desOs[*]{\EA} \s^*(t)+\desOs[*]{\EC}\desOs[*]{\yy}(t),
\end{align*}
where $\desOs[*]{\yy}$ can be understood as new input of the adjoint system.
Since~\eqref{eq:dual_sys} is a system with boundary control, the design of $\OKc$
follows immediately from that of $\BKc$ described in Section~\ref{sec:late_lumping_coc} by duality.
Finally, the observer gain $\Lc$ can be defined as the adjoint of the restriction of $\OKc$ to $D(\desOs[*]{\EA})$:
\begin{align*}
  \Lc^*\s^*(t) = \OKc\s^*(t) = \OR\s^*(t), && \s^*(t)\in D(\desOs[*]{\EA}).
\end{align*}

\subsection{Transformation based design}
\label{sec:trafo_based}
Following the design by dynamics assignment as described in Section~\ref{sec:late_lumping_coc},
in most cases the appropriate choice of $\desCs{\BR}$ is not obvious. Therefore, such designs
usually rely on  a state transformation $q(t) = \Trafo^q_\s \s(t)$.
In the new coordinates, the system
\begin{align*}
  \dot{q}(t) = {\BA}_q q(t), &&
  u(t) = \BR_q q(t)
\end{align*}
appears in a simplified form, where the choice of $\desCs{\BR}_q$ for the feedback
\begin{align}
  \label{eq:bc_system_ctrl}
\uu(t) = \BKc_q\, q(t) + \BG_q\desCs{\uu}(t),
\end{align}
with feedback gain
\begin{align*}
  \BKc_q=\BR_q - \BG_q\desCs{\BR}_q\in\LOp(D(\BA_q),\Is),
\end{align*}
is simple.
As described in Section~\ref{sec:late_lumping_coc},
this feedback assigns a desired dynamics $\desCs{\EA}_q$ to the
closed loop system
\begin{align}
  \label{eq:trafo_based_des_sys}
  \dot{q}(t) = \desCs{\BA}_q q(t), &&
  \desCs{\uu}(t) = \desCs{\BR}_q q(t).
\end{align}
The challenging part of  such designs is the determination of the transformation
$\Trafo^q_\s$, required to compute the feedback
\begin{align}
  \label{eq:trafo_based_ctrl}
\uu(t) = \BKc\, \s(t) + \BG_q\desCs{\uu}(t), && \BKc=\BKc_q\Trafo^q_\s\in\LOp(D(\BA),\Is)
\end{align}
in the original coordinates.

Typical examples of such designs are flatness-based designs (see, e.g., \cite{Woittennek2013cpde} and Section~\ref{sec:example_hyperbolic}) and simple
backstepping designs (cf.\ \cite{Krstic2008bcp}).
Similar techniques apply to the observer design \cite{Krstic2008bcp, Riesmeier2021pamm}.

\subsection{Properties of the involved operators}

The results of this article are restricted to desired closed-loop operators
$\desS{\EA}\in\{\desCs{\EA}, \desOs{\EA}\}$ which satisfy the following
assumption.
\begin{assum}
  \label{hypo:ad_spec}
  $\desS{\EA}$ has the following spectral properties.
\begin{enumerate}[ref={A\ref{hypo:ad_spec}.\arabic*},label={A\ref{hypo:ad_spec}.\arabic*:},leftmargin={3.9em}]
\item $\desS{\EA}$ is a Riesz-spectral
operator~\cite{GuoZwart2001report}.
\label{item:spec_riesz}
\item $\desS{\EA}$ is a discrete operator \cite{DunfordSchwartz3}.
\label{item:spec_disc}
\item The eigenvalues $\{\desS{\eval}_i\}_{i=1}^\infty$  of $\desS{\EA}$ are simple\footnote{Note that Assumption~\ref{item:spec_simple} is a reasonable technical assumption
in order to avoid the introduction of generalized eigenvectors and, this way, simplify computations.}.
\label{item:spec_simple}
\end{enumerate}
\end{assum}

Among others, from \ref{item:spec_riesz} it follows that the closure of the span of the eigenvectors $\{\desS{\evec}_i\}_{i=1}^\infty$
of $\desS{\EA}$ is $\Ss$, hence this sequence is well suited as an approximation basis for the state space.
\ref{item:spec_disc} ensures that $\desS{\EA}$ has a pure point spectrum
$\sigma(\desS{\EA})=\sigma_p(\desS{\EA})=\{\desS{\eval}_i\}_{i=1}^\infty$
without any finite accumulation points.

The eigenvectors $\{\desS{\evec}_i\}$ and $\{\desS[*]{\evec}_i\}$ of $\desS{\EA}$ and $\desS[*]{\EA}$
are assumed to be normalized, such that $\scal{\desS{\evec}_i}{\desS[*]{\evec}_i}=1$, $i\ge 1$.
Furthermore, $\scal{\desS{\evec}_i}{\desS[*]{\evec}_j}=0, \, i\neq j$ follows from Assumption~\ref{hypo:ad_spec}.
In order to use a perturbation result from \cite{XuSallet1996siam} the
input and output operators must satisfy the following condition.
\begin{assum}
  \label{hypo:xu_sallet_h3}
Let $d_i,\,i=1,2,...$ be the distance from the eigenvalue $\desS{\eval}_i\in\sigma(\desS{\EA})$ to the
rest of the spectrum $\sigma(\desS{\EA})$,
$\desS{D}_i=\{z\in\Compl\,|\,\frac{d_i}{3}>|z-\desS{\eval}_i|\}$ the disk centered at
$\desS{\eval}_i$ and
$\desS{D} = \bigcup_{i=1}^\infty \desS{D}_i$ the union of the disks.
\begin{enumerate}[ref={A\ref{hypo:xu_sallet_h3}.\arabic*},label={A\ref{hypo:xu_sallet_h3}.\arabic*:},leftmargin={3.9em}]
  \item For the elements $\{\desCs{\mb}_i=\dualdc{\desCs{\EB}}{\desCs[*]{\evec}_i}\}$
    of the modal input operator and the eigenvalues $\{\desCs{\eval}_i\}$ exists a $M\in\Reels^+$ such that
    \begin{align*}
    \sum_{i=1}^\infty \left|\frac{\desCs{\mb}_i}{\eval - \desCs{\eval}_i}\right|^2 \le M < \infty,
    && \forall \eval\not\in D=\desCs{D}.
    \end{align*}
  \label{item:xu_sallet_h3}
  \item For the elements $\{\desOs{\mc}_i=\desOs{\EC}\desOs{\evec}_i\}$
    of the modal output operator and the eigenvalues $\{\desOs{\eval}_i\}$ exists a $\desOs{M}\in\Reels^+$ such that
    \begin{align*}
    \sum_{i=1}^\infty \left|\frac{\desOs{\mc}_i}{\eval - \desOs{\eval}_i}\right|^2 \le \desOs{M} < \infty,
    && \forall \eval\not\in \desOs{D}.
    \end{align*}
  \label{item:xu_sallet_h3_dual}
\end{enumerate}
\end{assum}

Since $\desS{\EA}$ is
Riesz-spectral, the property $\sup_{i\in\Nats}\Re(\desS{\eval}_i) < +\infty$
implies that $\desS{\EA}$ generates a \CN-semigroup.
Furthermore, this semigroup has the form \cite{GuoZwart2001report}
\begin{align*}
\desS{\SG}(t) = \sum_{i=1}^\infty e^{\desS{\eval}_i t}\scal{\cdot}{\desS[*]{\evec}_i}\,\desS{\evec}_i.
\end{align*}
An important property for the stability analysis is, that the spectral
bound $s(\desS{\EA}) = \sup_{\eval\in\sigma(\desS{\EA})}\Re(\eval)$ of the
operator $\desS{\EA}$ coincides with the growth order\footnote{
The definition of the growth order implies that for each $\omega > \omega(\desS{\EA})$ there
exists a constant $M$ such that $\|\desS{\SG}(t)\|\le M e^{\omega t},\,t\ge 0$.
In the following the term growth order will also be used, for the growth order $\hat\omega$
of the time evolution $\s(t),\,t\ge 0$. In this case there exists a constant $M$,
such that for each $\omega>\hat\omega$, $\|\s(t)\|\le Me^{\omega t}\|\s_0\|,\,t\ge0$.}
$\omega(\desS{\EA})=\lim_{t\rightarrow\infty}\log\|\desS{\SG}(t)\|\,t^{-1}$ of
the semigroup $\desS{\SG}(t)$.
This property is called the spectrum determined growth
condition~\cite{XuFeng2001, CurZwa95}. It always holds for Riesz-spectral systems and
allows deducing exponential stability of the desired system from
$s(\desS{\EA})<0$.

In Section~\ref{sec:approx_fb} it will be shown that also the closed loop operator $\EAcl$,
obtained with an approximated controller,
generate a \CN-semigroup and
not only $\omega(\EAcl)=s(\EAcl)$, but also $\omega(\EAcl)=s_p(\EAcl)$ continues
to apply. That means that the growth bound of the \CN-semigroup is determined by the
bound $s_p(\EAcl)=\sup_{\clS{\eval}\in\sigma_p(\EAcl)}\Re(\clS{\eval})$ of the
point spectrum $\sigma_p(\EAcl)$.
In this case, exponential stability can be checked by computing
the eigenvalues of the closed loop system.
Moreover, it will be
shown that from some approximation order the spectrum of the closed loop system
converges to the desired spectrum.

\subsection{Further notation}
Depending on the context we refer to a certain system using
the operators
$(\BA^\star_\bullet, \BR^\star_\bullet, \EA^\star_\bullet, \EB^\star_\bullet, \EC^\star_\bullet, \SG^\star_\bullet,
\SysOpBcs[,\bullet]^\star, \SysOpBos[,\bullet]^\star)$
and the corresponding state space $\Ss_\bullet$. The superscript $\star$ defines the dynamics,
e.g. closed loop $\star=\clSText$ or desired $\star=\desSText$, and for fixed
$\star$ the subscript $\bullet$ determines the coordinates.
With this notation also other operators and
elements will be equipped in the next sections, for example,
the sequence of eigenvalues $\{\eval^\star_i\}$ of $\EA^\star_\bullet$,
the sequence of eigenvectors $\{\evec^{\star}_{\bullet,i}\}$ of $\EA^\star_\bullet$,
the sequence of eigenvectors $\{\evec^{\star *}_{\bullet,i}\}$ of $\EA^{\star *}_\bullet$
and the elements $\ms^{\star}_i=\scalf{\bullet}{\evec^{\star *}_{\bullet,i}}{\Ss_\bullet}$
of the modal state $\ms^{\star}$.
When referring to the open loop dynamics or to the
original coordinates $\s$ the respective placeholders $\star$ and $\bullet$ are left empty.

\section{Feedback approximation}
\label{sec:approx_fb}
The design methods described in the previous section result in a control law
$\uu(t)=\BKc\s(t)+\BG\,\desCs{\uu}(t)$, $\BKc\in\LOp(D(\BA),\Is)$ that may
include integral operators which have to be approximated for the purpose of implementation.
For ease of notation and without loss of generality in the following $\desCs{\uu}(t)=0$.

\subsection{Approximation scheme}
\label{sec:approx_fb_part_one}
With an approximation of $\s(t)$ by the convergent sequence $\{\s^n\}_1^\infty$, $\s^n\in\Ss$,
the control law reads
\begin{equation*}
\uu(t) = \BKc \lim_{n\rightarrow\infty}\s^n(t).
\end{equation*}
It is assumed that each element $\s^n$ of the approximating
sequence is an element of an $n$-dimensional subspace
$\Ss^n \subset \Ss$ of the state-space. Therefore, the
sequence $\{\Ss^n\}_{1}^\infty$ of finite dimensional subspaces of
$\Ss$ is considered, which has to be chosen in such a way, that for
each $\s\in\Ss$ a sequence
$\{\s^i\}_{1}^{n}$, $\s^i\in\Ss^i$ of approximations exists, that
converges to $\s$.
Finally, $\{\svec^n_i\}_1^n$ denotes
a basis of the space $\Ss^n$.  Thus, each element
$\s^n\in\Ss^n$ with the properties described above can be uniquely expressed by
\begin{align*}
  \s^n(t) = \sum_{i=1}^{n}\ms_i^n(t)\svec_i^n, && \ms_i^n(t)\in\Compl.
\end{align*}
Typically, the feedback operator $\BKc$ is unbounded on $\Ss$, so it does not commute with the limit,
$\BKc\lim_{n\rightarrow\infty}\s^n(t) \neq \lim_{n\rightarrow\infty}\BKc\s^n(t)$.
Therefore, an approximation requires a decomposition
\begin{equation}
  \label{eq:fb_decomp}
\BKc = \BKu + \Kb
\end{equation}
into an unbounded part $\BKu\in\LOp(D(\BA),\Is)$ and a bounded
part $\Kb\in\LOp(\Ss,\Is)$, where each $\Kb$ can be stated as $\Kb=\scal{k}{\cdot}$ with a suitable $k\in\Ss$.
As described above, $\BKu$ can not be approximated but is assumed to be of
simple structure, for example, a point evaluation at the boundary. Hence, it is
reasonable to assume that $\BKu$ can be realized exactly.
Note that, the desired structure of $\BKu$ is achieved
by choosing $\BG$ in~\eqref{eq:bcs_frak_fb} appropriately, c.f.~\cite{Woittennek2017ifac}.

With the decomposition~\eqref{eq:fb_decomp}, the boundary condition \eqref{eq:bc_system_input} can be written as
\begin{align}
\label{eq:ctrl_is_bc}
 \Kb\s(t) = \ctrlIs{\BR} \s(t), && \ctrlIs{\BR}=\BR - \BKu
\end{align}
and one can introduce the intermediate boundary control system $(\BA,\ctrlIs{\BR})$.
The operators $(\ctrlIs{\EA},\ctrlIs{\EB})$ of the corresponding state space representation
can be derived from $(\BA,\ctrlIs{\BR})$ in the same way as $({\EA},{\EB})$ from $(\BA,{\BR})$.
More precisely the dynamics operator
$\ctrlIs{\EA}: \Ss \supset D(\ctrlIs{\EA})\rightarrow \Ss$
of the intermediate system is given by
\begin{align*}
  \ctrlIs{\EA}h_\s=\BA\,h_\s , && h_\s\in D(\ctrlIs{\EA})=\{h \in D(\BA) | \ctrlIs{\BR}h=0\}
\end{align*}
and $\ctrlIs{\EB}\in D(\ctrlIs[*]{\EA})'$.
Now, the desired dynamics $\desCs{\EA}$
can also be written as
\begin{align*}
\desCs{\EA} = \ctrlIs{\EA} + \ctrlIs{\EB} \Kb.
\end{align*}

Note that the perturbation $\ctrlIs{\EB} \Kb$ of $\ctrlIs{\EA}$,
with bounded $\Kb$,
does not affect the domain of the adjoint operator,
i.e., $D(\desCs[*]{\EA})=D(\ctrlIs[*]{\EA})$.
Therefore, in contrast to $\desCs{\EA}=\EA+\EB\Kc$, the decomposition $\desCs{\EA}=\ctrlIs{\EA} + \ctrlIs{\EB} \Kb$
is well defined when viewing $\desCs{\EA}$ an operator $\Ss\to D(\ctrlIs[*]{\EA})'$.
Moreover, the input operators
$\ctrlIs{\EB},\desCs{\EB}\in D(\ctrlIs[*]{\EA})'=D(\desCs[*]{\EA})'$
coincide up to a scaling\footnote{
  According to~\eqref{eq:bc_system_input:desired}, the desired closed-loop
  system is independent of the choice of $\BG$.  In contrast, $\BG$ scales the
  controller intermediate system.  This is reflected by a scaling of the
  corresponding input operators.
}: $\desCs{\EB}=\BG\ctrlIs{\EB}$.

Due to the convergence of
$\{\s^{n}(t)\}$, the bounded part $\Kb\s(t)$ of the control law
can be written as a limit and approximated by choosing $n$ sufficiently large\footnote{
The difficult part of the late-lumping methods under consideration is the determination
of the bounded part $\Kb$, e.g. for backstepping designs the determination of the
backstepping kernel. For this reason, in~\cite{Woittennek2017ifac} the authors introduced
an approximation method which allows skipping this difficult part of the respective late-lumping design
if one is interested in a finite-dimensional approximation $\Kb^{n}$ of the resulting
bounded part only. Within this section no
understanding of the approximation method proposed in~\cite{Woittennek2017ifac} is required
since one can assume that $\Kb$ is explicitly available.
}:
\begin{align}
\label{eq:fa_ctrl_approx_scheme}
\Kb\s^{n}(t)&= \sum_{i=1}^{n}\ms_i^n(t)\Kb\svec_i^n = \Kb^n\s(t)
=\scal{x(t)}{k^n},
\end{align}
with $k^n\in\Ss^n$.
Now the closed-loop dynamics
$\clS{\EA}:\Ss\supset D(\clS{\EA})\rightarrow \Ss$ of the plant subject to
the approximated control law can be introduced:
\begin{align}
  \label{eq:def_clS_fb}
\clS{\EA} = \ctrlIs{\EA} + \ctrlIs{\EB} \Kb^{n}, && D(\clS{\EA})=\{h\in D(\BA)|\ctrlIs{\BR} h=\Kb^n h\}.
\end{align}

As the desired closed loop operator $\desCs{\EA}$ is Riesz-spectral by assumption,
the convergence of the closed-loop dynamics $\clS{\EA}$ to the desired one is
characterized in terms of the spectrum.
However, since $\ctrlIs{\EB}$ is unbounded, the convergence of the spectrum is not immediate,
i.e., $\Kb^n\to\Kb$ does not directly imply
$\sigma(\clS{\EA}) \to \sigma(\desCs{\EA})$ as $n\rightarrow\infty$.

\subsection{Well-posedness and convergence}
Within this section, it will be shown that the intermediate system
as well as the closed-loop system is well-posed. After that, it will be
proven that the closed-loop operator converges to the desired operator,
in a spectral sense.

To apply a perturbation result from \cite{XuSallet1996siam}, Hypotheses
H1-H3 from \cite{XuSallet1996siam} have to be fulfilled. The next
Lemma shows that \ref{item:xu_sallet_h3} implies H3 of \cite{XuSallet1996siam}.
\begin{lem}
\label{lem:xu_sallet_h3_2}
Assume that \ref{item:xu_sallet_h3} holds true, then
\begin{align*}
\sum_{i=1,i\neq j}^\infty \left|\frac{\desCs{\mb}_i}{\desCs{\eval}_j - \desCs{\eval}_i}\right|^2 \le 3M < \infty,
&& \forall j\in\Nats.
\end{align*}
\end{lem}
\begin{proof}
Let $\eval_{j,m}= \desCs{\eval}_j + \Delta\eval_{j,m},\,m\in\{1,2,3\}$ with
$|\Delta\eval_{j,m}|=\frac{d_j}{3}$ and $\arg(\Delta\eval_{j,m})=m\frac{2\pi}{3}$, be
three points on the boundary of $D_j$.
Note that, for each pair $(i,j)$ there is always a $m\in\{1,2,3\}$
such that $|\eval_{j,m} - \desCs{\eval}_i|<|\desCs{\eval}_j - \desCs{\eval}_i|$.
According to Assumption~\ref{hypo:xu_sallet_h3}
\begin{align*}
\sum_{i=1,i\neq j}^\infty\left|\frac{\desCs{\mb}_i}{\eval_{j,m} - \desCs{\eval}_i}\right|^2
< M, && \forall m\in\{1,2,3\}
\end{align*}
since $\eval_{j,m}\not\in D$.
Let $m_{j,i}$ be the index for which ${|\eval_{j,m_{j,i}} - \desCs{\eval}_i|}$ becomes minimal, for fixed $j$ and $i$.
Then
\begin{align*}
\sum_{i=1,i\neq j}^\infty \left|\frac{\desCs{\mb}_i}{\desCs{\eval}_j - \desCs{\eval}_i}\right|^2 &<
\sum_{i=1,i\neq j}^\infty \left|\frac{\desCs{\mb}_i}{\eval_{j,m_{j,i}} - \desCs{\eval}_i}\right|^2 \\ &<
\sum_{m=1}^3\sum_{i=1,i\neq j}^\infty\left|\frac{\desCs{\mb}_i}{\eval_{j,m} - \desCs{\eval}_i}\right|^2
< 3M.
\end{align*}
\end{proof}
Since H1 and H2 of \cite{XuSallet1996siam}
are also fulfilled, cf. Section~\ref{sec:intro}, we can apply the following
perturbation result, which is a part of \cite[Theorem~1]{XuSallet1996siam}.
\begin{lem}
\label{lem:xu_sallet_theorem11}
For arbitrary $\Kb\in\LOp(\Ss,\Is)$ the operator $\ctrlIs{\EA}=\desCs{\EA}-\ctrlIs{\EB} \Kb$ of the
controller intermediate system
is the generator of a \CN-semigroup,
is Riesz-spectral and has compact resolvent.
\end{lem}

Lemma~\ref{lem:xu_sallet_theorem11} is formulated in terms of $(\desCs{\EA},\ctrlIs{\EB},-\Kb;\ctrlIs{\EA})$,
but it can also be applied in terms of $(\desCs{\EA},\ctrlIs{\EB},-\Kb + \Kb^n;\clS{\EA})$.
Amongst others, this means that $\clS{\EA}$ generates a \CN-semigroup and $\omega(\clS{\EA})=s_p(\clS{\EA})$.
Finally, the convergence of $\clS{\EA}$ against $\desCs{\EA}$ in a
spectral sense is shown.
\begin{lem}
\label{lem:eps}
Let $D^\epsilon_i=\{z\in\Compl\,|\,|z-\desCs{\eval}_i|<\epsilon\frac{d_i}{3};\,\epsilon\in(0,1)\}$ be the disk centered at
$\desCs{\eval}_i$ and $D^\epsilon = \bigcup_{i=1}^\infty D^\epsilon_i$ the union of all such disks.
Under Assumption~\ref{hypo:xu_sallet_h3}
\begin{align}
\label{eq:lem_eps}
\sum_{i=1}^\infty \left|\frac{\desCs{\mb}_i}{\eval - \desCs{\eval}_i}\right|^2
< M\left(4 + \epsilon^{-2}\right) < \infty,
&& \forall\eval\not\in D^\epsilon.
\end{align}
\end{lem}
\begin{proof}
Inequality~\eqref{eq:lem_eps}
is fulfilled for $\eval\not\in D$ and it remains to proof the case $\eval\in D \setminus D^\epsilon$.
Consider the decomposition:
\begin{align}
\label{eq:d_eps_ineq_1}
\sum_{i=1}^\infty \left|\frac{\desCs{\mb}_i}{\eval - \desCs{\eval}_i}\right|^2 
&= \sum_{i=1,i\neq j}^\infty \left|\frac{\desCs{\mb}_i}{\eval - \desCs{\eval}_i}\right|^2
{+ \left|\frac{\desCs{\mb}_j}{\eval - \desCs{\eval}_j}\right|^2 }
\end{align}
for the case $\eval\in D_j \setminus D_j^\epsilon$.
First, it will be shown that the first term
can be majorized element-wise by
\begin{align}
\label{eq:eps_majorization}
\left|\frac{\desCs{\mb}_i}{\eval - \desCs{\eval}_i}\right|^2 \le 
4\inf_{\eval^*\in\partial D_j}\left|\frac{\desCs{\mb}_i}{\eval^* - \desCs{\eval}_i}\right|^2,
\end{align}
where $\partial D_j$ is the boundary of $D_j$.
To this end, let $\djimin=\frac{2}{3}d_j + \Delta_i$ and $\djimax=\frac{4}{3}d_j + \Delta_i$
be the minimal respectively the maximal distance from $\partial D_j$ to
$\desCs{\eval}_i$, with $\Delta_i=|\desCs{\eval}_i-\desCs{\eval}_j|-d_j$. The worst case estimate of \eqref{eq:eps_majorization}
can be made with $|\eval - \desCs{\eval}_i|\rightarrow\djimin$ and $|\eval^* - \desCs{\eval}_i|=\djimax$.
Hence, inequality \eqref{eq:eps_majorization} is satisfied for all
possible values $\Delta_i\in[0,+\infty)$, since
$
\djimax \le 2\,\djimin .
$

Back to decomposition
\eqref{eq:d_eps_ineq_1}, also the second term can be majorized using $|\eval-\desCs{\eval}_j|\ge \epsilon\frac{d_j}{3}$
and \ref{item:xu_sallet_h3} can be applied to both terms:
\begin{align*}
\sum_{i=1}^\infty \left|\frac{\desCs{\mb}_i}{\eval - \desCs{\eval}_i}\right|^2 
&\le 4\sum_{i=1,i\neq j}^\infty \left|\frac{\desCs{\mb}_i}{\eval^* - \desCs{\eval}_i}\right|^2
{+ \left|\frac{\desCs{\mb}_j}{\epsilon\,\frac{d_j}{3}}\right|^2 } \\
&< M \big( 4 + \epsilon^{-2}\big).
\end{align*}
\end{proof}
\begin{thm}
\label{thm:conv}
The spectra of $\clS{\EA}$ and $\desCs{\EA}$ are given by $\sigma(\clS{\EA})=\{\clS{\eval}_i\}_1^\infty$ and
$\sigma(\desCs{\EA})=\{\desCs{\eval}_i\}_1^\infty$.
Consider the sequence of disks $\{D^\epsilon_i\}_1^\infty$ from Lemma~\ref{lem:eps}.
For each $\epsilon\in(0,1)$,
there exists an $n_\epsilon$ such that for $n\ge n_\epsilon$,
the spectrum of $\clS{\EA}=\ctrlIs{\EA} + \ctrlIs{\EB} \Kb^{n}$
is close to $\sigma(\desCs{\EA})$ in the following sense.
\begin{enumerate}[ref={\alph*},label={(\alph*)}]
\item\label{item:conv_cont} $\sigma(\clS{\EA})$ is contained in the union
$D^\epsilon=\bigcup_{i=1}^\infty D_i^\epsilon$ of the disks $\{D^\epsilon_i\}_1^\infty$.
\item\label{item:conv_one} Each disk $D^\epsilon_i$ contains one and only one eigenvalue $\clS{\eval}_i\in\sigma(\clS{\EA})$
of simple algebraic multiplicity.
\end{enumerate}
\end{thm}
\begin{proof}
The proof is inspired by \cite[Lemma~3 and Lemma~4]{XuSallet1996siam}.
Let $\tilde k^n=k-k^n$ be the approximation error of the feedback $\Kb=\scal{\cdot}{k}$,
then the closed loop operator can be stated as $\clS{\EA}=\desCs{\EA} + \Delta\desCs{\EA}$
with $\Delta\desCs{\EA}=-\ctrlIs{\EB}\scal{\cdot}{\tilde k^n}$.

(\ref{item:conv_cont}):
It is sufficient to show that 
the characteristic function
of $\clS{\EA}$, has no zeros for $\eval\not\in D^\epsilon$.
According to \cite{XuSallet1996siam}
\begin{align*}
g(\eval,n)=1 - \dualc{\ctrlIs{\EB}}{R(\eval, \desCs[*]{\EA})\tilde k^n}, && \eval\not\in\sigma(\desCs[*]{\EA})
\end{align*}
is a characteristic function of $\clS{\EA}$.
Moreover $\tilde k^n$
can be represented by the eigenvectors $\{\desCs[*]{\evec}_i\}_1^\infty$
of $\desCs[*]{\EA}$:
$\tilde k^n=\sum_{i=1}^\infty \desCs[,n]{\tilde{k}}_i\desCs[*]{\evec}_i$.
Then the resolvent $R(\eval,\desCs[*]{\EA})$ of $\desCs[*]{\EA}$
can be expanded into the series
\begin{align*}
R(\eval,\desCs[*]{\EA}) = \sum_{i=1}^\infty \frac{1}{\eval - \desCs{\eval}_i}\scal{\cdot}{\desCs{\evec}_i}\,\desCs[*]{\evec}_i
\end{align*}
and the characteristic function can be written as
\begin{align*}
g(\eval,n)=1 - g_{\epsilon}(\eval,n), &&
g_{\epsilon}(\eval,n) = \sum_{i=1}^\infty \frac{\desCs[,n]{\tilde{k}}_i \, \desCs{b}_i}{\eval - \desCs{\eval}_i}.
\end{align*}
Applying the Cauchy-Schwarz inequality to $g_\epsilon(\eval)$
\begin{align*}
|g_\epsilon(\eval,n)|^2
\le  \sum_{i=1}^\infty {\left|\frac{\desCs{b}_i}{\eval - \desCs{\eval}_i}\right|^2}
\sum_{i=1}^{\infty}{\left|\desCs[,n]{\tilde{k}}_i\right|^2}
\end{align*}
and
using Lemma~\ref{lem:eps}
it becomes clear that there is always an $n_\epsilon$ such that 
\begin{align*}
|g_\epsilon(\eval,n)|
\le  \sqrt{M_\epsilon
\sum_{i=1}^{\infty}\left|\desCs[,n]{\tilde{k}}_i\right|^2}
< 1, && n\ge n_\epsilon,
\end{align*}
since one can always find a constant $M_\epsilon \le M(4+\epsilon^{-2})$
and with $n\rightarrow \infty$ also $\|\tilde k^n\|_\Ss^2=\scal{\tilde k^n}{\tilde k^n}\rightarrow 0$
respectively $\sum_{i=1}^\infty |\desCs[,n]{\tilde{k}}_i|^2\rightarrow 0$.
As a consequence $|g(\eval,n)|>0$ for $n\ge n_\epsilon,\, \eval\not\in D_\epsilon$
and, thus, $\sigma(\clS{\EA})\subset D_\epsilon$.

(\ref{item:conv_one}): Consider $\eval\in D_j^\epsilon$ for fixed $j$ and
$n\ge n_\epsilon$.
Multiplying $g_\epsilon(\eval,n)$ and $g_1(\eval)=1$ by $\eval-\desCs{\eval}_j$,
it follows from the Rouché theorem for holomorphic functions that $g(\eval,n)$ has exactly
one root inside $D_j^\epsilon$, since $|g_\epsilon(\eval,n)|<|g_1(\eval)|$
on the boundary of $D_j^\epsilon$.
\end{proof}

\begin{rem}
Under the assumption that $\ctrlIs{\EB}$ is admissible, i.e., $\{\desCs{\mb_i}\}\in\li$ \cite{XuSallet1996siam},
and some minor modifications of Lemma~\ref{lem:eps} and Theorem~\ref{thm:conv}, one could replace
the disks $\{D^\epsilon_i\}$ with radius $\epsilon\frac{d_i}{3}$, $\epsilon\in(0,1)$,
with the disks $\{\check D^\epsilon_i\}$ with radius $\min(\epsilon,\frac{d_i}{3})$, $\epsilon\in(0,+\infty)$.
This can be useful in cases where $\{d_i\}\not\in\li$.
\end{rem}

\section{Observer gain approximation}
\label{sec:approx_obs}
To derive an appropriate approximation scheme for the observer gain
$\Lc$, the dual result to the one derived in Section~\ref{sec:approx_fb} is
developed. To this end a decomposition $\Lc=\Lu+\Lb$
into an unbounded part $\Lu\in D(\desOs{\EA})'$ and a bounded part $\Lb\in\Ss$ is
required, which will be explained in more detail below.

Often, the design by dynamics assignment of $\Lc$ is done directly for the primal system with boundary observation~\eqref{eq:bo_system} and not for the
adjoint system with boundary control~\eqref{eq:dual_sys}.
However, it is not straightforward to give a general expression for the domain
of the primal system operator.  Therefore, to provide an approximation scheme
for the observer gain, it is more convenient to do this with respect to the
adjoint system.
Nevertheless, for the application of the resulting approximation scheme
it doesn't matter which way the observer gain was derived.

As described in Section~\ref{sec:late_lumping_ooc},
the control law $\yy^*(t) = \OKc\s^*(t)$ assigns the desired dynamics $\desOs{\EA}$
to the boundary control system
$\dot{\s}^*(t)=\SysOpBos^*(\s^*(t),\yy^*(t))$.
Following Section~\ref{sec:approx_fb_part_one}, the feedback gain
has a decomposition
\begin{align*}
\OKc = \OKu + \Lb^*, && \OKu\in\LOp(D(\OA),\Os), && \Lb^*\in\LOp(\Ss,\Os)
\end{align*}
into an unbounded part $\OKu$ and a bounded
part $\Lb^*$.

With respect to the controlled system
\begin{align}
  \label{eq:adj_controlled_bcs}
\dot{\s}^*(t)=\SysOpBos^*(\s^*(t),(\OKu + \Lb^*)\s^*(t))
\end{align}
the adjoint of the observer intermediate system
\begin{multline*}
  \obsIs[*]{\SysOpBos}(\s^*(t),\Lb^*\s^*(t)) = 
  \SysOpBos^*(\s^*(t),(\OKu + \Lb^*)\s^*(t)), \\
   D(\obsIs[*]{\SysOpBos})=\{(h_{\s^*},h_{\yy^*})\in D(\OA)\times\Os \,|\, \obsIs{\OR} h_{\s^*}=h_{\yy^*} \},
\end{multline*}
with $\obsIs{\OR}=\OR - \OKu$,
and the adjoint of the desired observer system
\begin{multline*}
  \desOs[*]{\SysOpBos}(\s^*(t),0) = 
  \obsIs[*]{\SysOpBos}(\s^*(t),\Lb^*\s^*(t)), \\
   D(\obsIs[*]{\SysOpBos})=\{(h_{\s^*},h_{\yy^*})\in D(\OA)\times\Os \,|\, \BG\,\desOs{\OR} h_{\s^*}=h_{\yy^*} \},
\end{multline*}
with $\BG\,\desOs{\OR}=\OR - \OKc$,
can be defined.
Therewith, the adjoint desired system has the following state space representation:
\begin{align*}
  \dot{\s}^*(t) = \desOs[*]{\EA}\s^*(t) = (\obsIs[*]{\EA}+\obsIs[*]{\EC}\Lb^*)\s^*(t), && \obsIs[*]{\EC}\in D(\obsIs{\EA})'
\end{align*}
with
\begin{multline*}
  \desOs[*]{\EA}: \Ss\supset D(\desOs[*]{\EA}) \to \Ss, \\
  D(\desOs[*]{\EA}) = \{h\in D(\OA)\,|\, \desOs{\OR}=0\}
\end{multline*}
and
\begin{align*}
  \obsIs[*]{\EA}: \Ss\supset D(\obsIs[*]{\EA}) \to \Ss, 
  D(\obsIs[*]{\EA}) = \{h\in D(\OA)\,|\, \obsIs{\OR}=0\}.
\end{align*}
The observer intermediate dynamics $\obsIs[*]{\EA}$ is formally given by $\EA^*+\EC^*\Lu^*$,
with
\begin{align*}
  \Lu^*\s^*(t) = \OKu\s^*(t) = \OR\s^*(t), && \s^*(t)\in D(\obsIs[*]{\EA}).
\end{align*}
The primal operator $\desOs{\EA}=\obsIs{\EA}+\Lb\obsIs{\EC}$ is now
defined in terms of the adjoint operator $\obsIs[*]{\EA}+\obsIs[*]{\EC}\Lb^*$.
Therein, $\Lb\in\Ss$, defined by the relation
$\Lb^* = \scal{\cdot}{\Lb}$, can be expanded into the convergent
series:
\begin{align*}
\Lb=\lim_{n\to\infty}\sum_{i=1}^n l_i^n\,\svec^n_i\,\in\,\Ss, && l_i^n\in\Compl.
\end{align*}

Since the perturbation $\Lb \obsIs{\EC}$ of $\obsIs{\EA}$,
with $\Lb\in\Ss$, does not affect the domain of the operator, $D(\obsIs{\EA})=D(\desOs{\EA})$
and $\desOs{\EC}=\BG\,\obsIs{\EC}$.

The required observer can now be derived from the original system dynamics and
the desired observer error dynamics, both written
in terms of the observer intermediate dynamics:
\begin{subequations}
\label{eq:obs_only_cl_sys}
\begin{align}
\label{eq:obs_only_cl_sys_sys}
\dot{\s}(t) &= \EA\s(t) = \obsIs{\EA}\s(t) - \Lu\yy(t)\\
\label{eq:obs_only_cl_sys_err}
\dot{\st}(t) &= \desOs{\EA} \st(t)
= \obsIs{\EA} \st(t) + \Lb\yt(t),
\end{align}
\end{subequations}
with $\yt(t)=\yh(t) - \yy(t)$
and $\st(t)=\sh(t)-\s(t)$.
From \eqref{eq:obs_only_cl_sys}, the observer
\begin{align}
\label{eq:obs_only_cl_sys_obs}
\dot{\sh}(t) &= \obsIs{\EA}\sh(t) - \Lu\yy(t) + \Lb \yt(t)
\end{align}
can be derived.

The observer gain can now be approximated\footnote{
  As for the controller approximation, the determination of $\Lb$ is the most
sophisticated part during the late-lumping design. The authors provided
two modal approximation schemes for $\Lb$, one for a simple analytic system~\cite{Riesmeier2018cdc}
and one for a class of hyperbolic systems~\cite{Riesmeier2021pamm}, cf. Section~\ref{sec:ana} and Section~\ref{sec:hyp}. With these approximation
schemes, the determination of $\Lb$ can be skipped if one is only interested in
the finite-dimensional approximation $\Lb^{n}$.
Within this section no
understanding of these approximation schemes is required
since one can assume that $\Lb$ is explicitly available.} by
series truncation, means, replacing $\Lb$ with
\begin{align}\label{eq:oa_convergent_l}
\Lb^{n}=\sum_{i=1}^n l_i^n\,\svec^n_i\,\in\,\Ss
\end{align}
in \eqref{eq:obs_only_cl_sys}-\eqref{eq:obs_only_cl_sys_obs}.

Therewith, the growth rate of the observer error $\st(t)=\sh(t)-\s(t),\,\,t\ge 0$
is determined by $s_p(\obsIs{\EA}+\Lb^n\obsIs{\EC})$, which is a consequence of
the following lemma.

\begin{lem}
\label{lem:conv_obs}
$\obsIs{\EA} + {\Lb}^n\obsIs{\EC}:\Ss\supset D(\obsIs{\EA})\rightarrow \Ss$ is the
generator of a \CN-semigroup, is Riesz-spectral, has compact resolvent, and
$\sigma(\obsIs{\EA} + {\Lb}^n\obsIs{\EC})$
converges to $\sigma(\desOs{\EA})$ in the sense of Theorem~\ref{thm:conv},
as $n\rightarrow\infty$.
\end{lem}
\begin{soproof}
Note that the results in \cite[Theorem~1]{XuSallet1996siam}
are not directly derived for the closed-loop operator $\clS{\EA}$
of Section~\ref{sec:approx_fb} but, in an intermediate step, for its adjoint
\begin{align*}
\clS[*]{\EA}=\desCs[*]{\EA}-\tilde k^n\dualc{\ctrlIs{\EB}}{\cdot}:\,D(\ctrlIs[*]{\EA})\subset\Ss\rightarrow \Ss.
\end{align*}
Obviously, $\obsIs{\EA} + {\Lb}^n\obsIs{\EC}$ possesses the same structure
with $\tilde k^n\in\Ss$ replaced by $\Lb^n\in\Ss$ and 
$\dualc{\ctrlIs{\EB}}{\cdot}\in\LOp(D(\ctrlIs{\EA}),\Is)$ replaced by $\obsIs{\EC}\in\LOp(D(\obsIs{\EA}),\Os)$.
Therefore, in view of Assumption~\ref{item:xu_sallet_h3_dual} and in 
accordance with Lemma~\ref{lem:xu_sallet_theorem11} and Theorem~\ref{thm:conv},
the results obtained for the closed-loop operator $\clS{\EA}$ in Section~\ref{sec:approx_fb} directly translate to $\obsIs{\EA} + {\Lb}^n\obsIs{\EC}$.
\end{soproof}

\section{Analytic case}
\label{sec:ana}

Up to here the theory and the approximation schemes, are developed for the
abstract system class introduced in Section~\ref{sec:prelim}.
To give further insights in terms of the application of the results,
within this section
the implementation of the presented approximation schemes
will be sketched in the context of a class of so-called analytic systems.
This system class will be introduced in the following as the \ac{ac}.

\begin{defn}[The analytic case]
  \label{def:par_case}
Motivated by \cite[Example~2.18]{CurZwa95}, \cite[Property~P4]{Deutscher2013IJC},
within this contribution, one speaks of the \acf{ac},
if for $c>0,\, \omega\in\Reels$ each $\eval\in\sigma(\desS{\EA})$ lies in the
sector $|\text{Im}(\eval)| \le c (\omega - \text{Re}(\eval))$ of the complex plane.
\end{defn}

The sector condition ensures that the operator $\desS{\EA}$ is
analytic. Depending on the literature, the term analytic operator is synonymous with
holomorphic operator or sectorial operator.
Especially the dynamics of the important class of parabolic/diffusion systems can be
described using this type of operator.

In the \ac{ac}, it is in many cases simple to prove Assumption~\ref{hypo:ad_spec},
for example, if $\desCs{\EA}$ can be rewritten in terms of a Sturm-Liouville operator.
Furthermore, if Assumption~\ref{hypo:ad_spec} applies
and the system is of parabolic type ($|\desCs{\eval}_i| \propto i^2$) Assumption~\ref{hypo:xu_sallet_h3}
applies too, at least for admissible input operators $\desCs{\EB}$ ($\{\desCs{\mb}_i\}\in\li$).
But even if $\desCs{\EB}$ is not admissible Assumption~\ref{hypo:xu_sallet_h3}
continues to apply in this case, as long $\desCs{\EB}$ is not ``too unbounded''.
A simple example is the 1-D heat equation with Dirichlet boundary condition
at one boundary and Dirichlet actuation at the other boundary \cite[Example~1 for $\beta=0$]{HoRussel1983}, which is not
admissible but still satisfies Assumption~\ref{hypo:xu_sallet_h3}.

Instead of discussing a detailed example\footnote{Due to a matter of space only one example will be discussed in detail,
which this time is a hyperbolic one, see Section~\ref{sec:hyp}.}, the reader is referred to \cite{Woittennek2017ifac} and \cite{Riesmeier2018cdc}.
The example in \cite{Woittennek2017ifac} consists of a reaction-diffusion-system with a homogeneous Robin boundary
condition at the one boundary, a Neumann actuation on the other boundary and
constant coefficients. The backstepping controller design and approximation
is also treated in \cite{Woittennek2017ifac}, and the backstepping observer
design and approximation can be taken from \cite{Riesmeier2018cdc}.
Only the consideration of the unbounded part of the observer gain has to be
adjusted according to Section~\ref{sec:approx_obs} since it has not been emphasized
in~\cite{Riesmeier2018cdc}. As described in the
Sections~\ref{sec:approx_fb}-\ref{sec:approx_obs}, the stability analysis for both configurations,
can be done by computing the eigenvalues of the closed-loop system.

Other examples for the  \ac{ac}, where late-lumping design is more involved,
are diffusion systems with spatially varying coefficients~\cite{Krstic2008bcp}
or in-domain actuation~\cite{WoittennekWang2014ifac}.

\section{Hyperbolic case}
\label{sec:hyp}

As described in Section~\ref{sec:hyp} for the \ac{ac},
within this section the application of the results of this article will be discussed
in the context of a particular class of hyperbolic systems.
This class will be termed in the following as the \ac{hc}.
Unlike the \ac{ac}, a uniform controller and observer design exists
for the \ac{hc}. Therefore, it is possible to specify general
design parameters for the \ac{hc}, such that the necessary assumptions
for the application of the results derived within this article are fulfilled.

To avoid the necessity of determining the bounded parts $\Kb$ and $\Lb$
of the gains $\Kc$ and $\Lc$
explicitly, the respective design will be simplified using the
late-lumping design technique introduced in~\cite{Woittennek2017ifac,Riesmeier2021pamm},
to directly derive the approximations $\Kb^n$ and $\Lb^n$.

The section is divided into two parts. First,
Subsection~\ref{sec:hyp_gen_aspects} provides the definition
of the \ac{hc} and the necessary background for the
implementation of the approximation schemes.
Second, Subsection~\ref{sec:example_hyperbolic}
shows in detail the application of the results obtained so far.

\subsection{General aspects}
\label{sec:hyp_gen_aspects}
In this subsection, the \ac{hccf} and the \ac{hocf} will be recalled,
and, based on this, the \ac{hc} will be defined.
Furthermore, the related approximation schemes for the state-feedback gain and
the observer output-injection gain will be presented.

\subsubsection{Hyperbolic controller canonical form}
\label{sec:hyp_hccf}

For several hyperbolic systems of the form~\eqref{eq:bc_system}, there exists a bounded
invertible map $\Trafo_\s^\scc:\Ss\to \Ss_\scc$, i.e. a state transformation 
\begin{equation*}
\scc = \Trafo_\s^\scc\, \s \in \Ss_\scc = \Compl^N\times \Lz(\thu, \tho;\Compl),
\end{equation*}
such that in new coordinates $\scc$ the system appears in the \ac{hccf} \cite{Russell1991jiea, WoittennekRudolph2012mathmod}.
The \ac{hccf} describes a system of differential equations consisting of a chain of integrators which is attached to the output of a transport system,
where the system input corresponds to the input of the transport system up to a feedback.
\begin{defn}[\Acl{hccf}]
  \label{def:hccf}
The \ac{hccf} is defined as the following boundary control system
\begin{subequations}
\label{eq:hccf_bc_system}
\begin{align}
\label{eq:hccf_bc_system_dynamics}
\dot{\scc}(t) &= \BA_\scc \scc(t), && \scc(0) = \scc_0 \in \Ss_\scc\\
\label{eq:hccf_bc_system_input}
\uu(t) &= \BR_\scc \scc(t)
\end{align}
with the differential operator
\begin{align}
\BA_\scc h &= (h_2, ..., h_N, h_{N+1}(\thu), \diff{\theta}h_{N+1}) \\
D(\BA_\scc) &= \{h\in\Ss_\scc\,|\,\diff{\theta}h_{N+1}\in\Lz(\thu,\tho;\Compl)\},
\end{align}
the boundary operator
\begin{align}
\label{eq:hccf_bc_bc}
\BR_\scc h &= h_{N+1}(\tho) + \dualf{\acc}{\cconj{h}}{\BA_\scc}, \,\, h\in D(\BA_\scc)\\
\label{eq:hccf_bc_a}
\dualf{\acc}{\cconj{h}}{\BA_\scc} &= \sum_{i=1}^m \accu_i h_{N+1}(\theta_i) + \scale{\accb}{\cconj{h}}, \,\, m\in\Nats 
\end{align}
\end{subequations}
and $\thu = \theta_m \le ... \le \theta_1 < \tho$.
\end{defn}

The corresponding system operator $\EA_\scc$ is given by
\begin{align*}
  \EA_\scc h_\scc = \BA_\scc h_\scc, && h_\scc\in D(\EA_\scc)\!=\!\{h\in D(\BA_\scc)|\BR_\scc h = 0\},
\end{align*}
as described in Section~\ref{sec:prelim:sys_struct}.

  The state variable $\scc_1$ of the \ac{hccf} constitutes a flat output of the given system.
  This can be easily verified as all state variables can be expressed by $\scc_1$, using only time derivatives
  and predictions:
  \begin{align}
    \label{eq:conn_scc_scco}
  \scc(t) = \Big(\scc_1(t),\dot{\scc}_1(t),...,{\scc}_1^{(N-1)}(t),{\scc}_1^{(N)}(t-\theta_-+\cdot)\Big).
  \end{align}
  In many cases, like in the case discussed in \cite{Woittennek2013cpde}, it is possible to determine a flat output\footnote{
    Obviously, $\Phi$ is bounded on $\Ss$ if $N>0$.
  }
  \begin{align}
    \label{eq:conn_fc_cc}
  \fo(t)=\Phi\s(t)=r\,\scc_1(t-\thu), && \Phi\in\LOp( D(\EA),\Is)
  \end{align}
  directly from the original system~\eqref{eq:system}, although this differs from $\scc_1$ by a scaling $r$ and a time shift $\thu$.
  This a-priory knowledge can be used to compute the required state transform to the \ac{hccf},
since the restriction of the flat-output trajectory
$\fo(t)=r\,\scc_1(t-\thu)$ to the interval $t\in(\thu,\tho)$,
defines a state of the system:
\begin{align}
  \label{eq:def_nu_state}
\sfc(t)=\fo(t+\cdot)=\Hn{N}(\thu,\tho;\Compl).
\end{align}
The corresponding state transformation $\sfc(t)=\Trafo_\scc^\sfc \scc(t)$
can be determined by expressing $\scc_1(t-\thu-\theta)$ on the interval $\theta\in (\thu,\tho)$
in terms of $\scc(t)$.

\begin{rem}
  \label{rem:hccf_russel}
It can be observed that the functional $\acc\in D(\BA_\scc)'$,
which determines the dynamics,
can be decomposed into an unbounded part
associated with the point evaluations at $\theta_1, ..., \theta_m\in[\thu,\tho)$
and a bounded part $\accb=(\accb_1,...,\accb_N,\accb_{N+1}(\cdot))\in\Ss_\scc$.
  In fact, functional $\acc$ defined in \eqref{eq:hccf_bc_a},
  could be
  replaced by the even more general form
  \begin{align*}
    \dualf{\acc}{\cconj{h}}{\BA_\scc} = \accu_m h_{N+1}(\thu) + \int_{\thu}^{\tho} h_{N+1}(\theta)\,dv(\theta),
  \end{align*}
  where $v(\theta)$ is a function of bounded variation,
  satisfying additional conditions at the boundaries,
  cf. Russel~\cite[pp.~136-137]{Russell1991jiea}. However, the given form~\eqref{eq:hccf_bc_a} covers various practical relevant dynamics and no additional theory has to
  be introduced.
\end{rem}

\begin{defn}[The hyperbolic case]
  \label{def:hyp_case}
Within this contribution one speaks of the \acf{hc}
if the desired system $(\desCs{\BA},\desCs{\BR})$ resp.\ $(\desOs{\OA},\desOs{\OR})$ can be transformed into the
\ac{hccf}, cf.~\cite{WoittennekRudolph2012mathmod}, Definition~\ref{def:hccf} and Remark~\ref{rem:hccf_russel}.
\end{defn}

In the \ac{ac}, it is, in many cases, simple to proof Assumption~\ref{hypo:ad_spec}, by rewriting
$\desCs{\EA}$ in terms of a Sturm–Liouville operator.
In the \ac{hc}, it can be more difficult to proof/ensure Assumption~\ref{hypo:ad_spec} for $\desCs{\EA}$.
Therefore, in the next subsection, simple conditions
will be provided such that Assumption~\ref{hypo:ad_spec} holds in the \ac{hc}.

\subsubsection{Controller design}
\label{sec:hcc_ctrl_design}
Consider the \ac{hccf} according to Definition~\ref{def:hccf}.
With the feedback
\begin{align}
  \label{eq:hccf_controller}
\uu(t) &= \BKc_\scc\scc(t) = (\BR_\scc - \BG_\scc\desCs{\BR}_\scc)\scc(t),
\end{align}
where
\begin{align*}
\desCs{\BR}_\scc h &= h_{N+1}(\tho) + \dualf{\desCs{\acc}}{\cconj{h}}{\BA_\scc}, \,\, h\in D(\BA_\scc)\\
\dualf{\desCs{\acc}}{\cconj{h}}{\BA_\scc} &= \sum_{i=1}^m \desCs{\accu}_i h_{N+1}(\theta_i) + \scale{\desCs{\accb}}{\cconj{h}}, \,\, m\in\Nats,
\end{align*}
the system dynamics $\EA$ can be compensated and the desired dynamics $\desCs{\EA}$ is
achieved for the closed loop system. Although it is difficult to characterize the set of all possible
$\desCs{\acc}$ such that the closed loop system is exponentially stable,
a particular $\desCs{\acc}$ is given in~\cite{Woittennek2012at} that is derived from
a stable delay differential equation. This one will be discussed later, cf.~\eqref{eq:hcc_delay_diff}.

In view of the results of this article, it is necessary to choose $\desCs{\acc}$
such that $\desCs{\EA}_\scc$ generates a stable semigroup and
Assumption~\ref{hypo:ad_spec} and \ref{item:xu_sallet_h3} are valid in terms of $(\desCs{\EA}_\scc,\desCs{\EB}_\scc)$.
\ref{item:spec_riesz} holds, since according to~\cite{Russell1991jiea}
$\desCs{\EA}_\scc$ is a Riesz-spectral operator
as long as $\desCs{\accu}_m\neq 0$.
To ensure \ref{item:spec_simple} and \ref{item:xu_sallet_h3} it is useful that
the eigenvalue asymptotics of $\desCs{\EA}_\scc$,
given by $\sigma(\desCs{\EA}_\scc|_{\desCs{\accb}=0})$,
fulfils\ref{item:spec_simple} and \ref{item:xu_sallet_h3}, too.
Depending on whether the delays $\tau_i=\theta_i-\tho,\,i=1,...,m$
are commensurate\footnote{Commensurate means, that the delays have a representation
$\tau_i=n_i \tau_0,\,i=1,...,m$, with $n_i\in\Nats,\,i=1,...,m$ and $\tau_0\in\Reels^+$~\cite{MichielsNiculescu2007}.},
the eigenvalues of $\desCs{\EA}_\scc|_{\desCs{\accb}=0}$ are not simple or
can approach each other arbitrarily close.
Both cases can be avoided~\cite{MichielsNiculescu2007} by restricting the attention to the case $\desCs{\accu}_i=0, \,i=1,...,m-1$.
In this case, $\desCs{\EA}_\scc$ is a discrete spectral operator\footnote{
Since $\Ss_\scc$ and the product space used in \cite{Rabah2003crm} are not isometric isomorphic,
there is no state transformation between their elements,
and we can not directly infer properties of $\EA_\scc$ from the generator in \cite{Rabah2003crm}.
But, thanks to \cite[Theorem~2.1.10, Excercise~2.16]{CurZwa95},
the spectrum of $\EA_\scc$ and the spectrum of the generator in \cite{Rabah2003crm} coincide.
}, cf.~\cite[Proposition~2.2]{Rabah2003crm}.
For $\desCs{\EA}_\scc$ stable $\desCs{\accu}_m\in(-1, 1)$ is a necessary condition~\cite{MichielsNiculescu2007}.
But it is again difficult to characterize the general form of the remaining bounded part $\desCs{\accb}$,
such that the resulting dynamics are stable, \ref{item:spec_simple} remains valid, and \ref{item:xu_sallet_h3} is not violated.

To derive a specific $\desCs{\accb}$ the above-mentioned $\desCs{\acc}$, derived from the stable delay differential equation
\begin{align}
  \label{eq:hcc_delay_diff}
  \sum_{i=0}^N \kappa_i\big(\fo^{(i)}(t+\tho) + \mu\fo^{(i)}(t+\thu)\big) = 0,
\end{align}
for the flat output $\fo(t)$, is considered.
It is readily seen that $\desCs{\accu_m} = \mu$ and $\desCs{\accu_i} = 0,\, i=1,...,m-1$.
An explicit formula for $\desCs{\accb}$ can be taken from~\cite[Equation~11]{Woittennek2012at_eng}.
The associated spectrum is composed of a finite and an infinite part: $\sigma(\desCs{\EA}_\scc)=\{\desCs{\eval}_{\kappa,i}\}_1^N\cup\{\desCs{\eval}_{\mu,i}\}^{\infty}_{-\infty}$.
The infinite part $\{\desCs{\eval}_{\mu,i}\}^{\infty}_{-\infty}$ is given by
\begin{align*}
 \desCs{\eval}_{\mu,i} = \Delta\theta^{-1}\left\{
  \begin{array}{ll}
     \ln|\mu| + \mathrm j\, 2 i\pi,\\
     \ln|\mu| +  \mathrm j\, (2 i - 1)\pi,
  \end{array}
   \right.
   &&
  \begin{array}{ll}
     \mu < 0 \\
     \mu > 0,
  \end{array}
\end{align*}
with $\Delta\theta = \tho-\thu$, $i\in\Integers$ and $\mathrm j$ the imaginary unit.
The finite part $\{\desCs{\eval}_{\kappa,i}\}_1^N$ is determined by
the zeros of the polynomial $\sum_{i=0}^N \kappa_i \eval^i$. It can
be placed distinctly from the infinite part as a set of $N$ simple eigenvalues.
\begin{thm}
  \label{thm:hccf_ad}
  Let $\desCs{\EA}_\scc$ be the operator in \ac{hccf}, derived from the delay differential equation~\eqref{eq:hcc_delay_diff},
  with simple eigenvalues $\{\desCs{\eval}_{\kappa,i}\}_1^N$, which are distinct from $\{\desCs{\eval}_{\mu,i}\}^{\infty}_{-\infty}$.
  Then \ref{item:xu_sallet_h3} holds in terms of $(\desCs{\EA}_\scc,\desCs{\EB}_\scc)$.
\end{thm}
\begin{proof}
  The input operator is admissible and therewith $\{\desCs{\mb}_i\}\in\li$, cf.~\cite{Russell1991jiea}.
  First, we exploit the admissibility.
  With $\desCs{\mb}_{\sup} = \sup_{i\in\Nats}\desCs{\mb}_i$
  \begin{align*}
    \sum_{i=1}^\infty \left|\frac{\desCs{\mb}_i}{\eval - \desCs{\eval}_{i}}\right|^2 \le
    \sum_{i=1}^\infty \left|\frac{\desCs{\mb}_{\sup}}{\eval - \desCs{\eval}_{i}}\right|^2.
  \end{align*}
  From the decomposition
  \begin{align*}
    \sum_{i=1}^\infty \left|\frac{\desCs{\mb}_{\sup}}{\eval - \desCs{\eval}_{i}}\right|^2 =
    \sum_{i=1}^N \left|\frac{\desCs{\mb}_{\sup}}{\eval - \desCs{\eval}_{\kappa,i}}\right|^2 +
    \sum_{i=-\infty}^\infty \left|\frac{\desCs{\mb}_{\sup}}{\eval - \desCs{\eval}_{\mu,i}}\right|^2
  \end{align*}
  it is obvious that, for the finite part of the sum, a constant $M_\kappa\in\Reels^+$ exists such that 
  \begin{align*}
    \sum_{i=1}^N \left|\frac{\desCs{\mb}_{\sup}}{\eval - \desCs{\eval}_{\kappa,i}}\right|^2 \le M_\kappa < \infty,
    && \forall \eval\not\in D,
  \end{align*}
  with $D$ according to \ref{item:xu_sallet_h3}.
  Consider the sequence $\{|\eval-\desCs{\eval}_{\mu,i}|\}_1^\infty$
  and the sequence  $\{\frac{d}{4}i\}_1^\infty$,
  with the distance $d=2\pi\Delta\theta^{-1}=|\desCs{\eval}_{\mu,i} - \desCs{\eval}_{\mu,i+1}|$ between two adjacent eigenvalues.
  It becomes clear that for each $\eval\not\in D$ there is at least one
  reordering of $\{\desCs{\eval}_{\mu,i}\}_1^\infty$, such that
  \begin{align*}
  |\eval-\desCs{\eval}_{\mu,i}| \ge \frac{d}{4}i, && i\ge 1.
  \end{align*}
  This implies the estimate
  \begin{align*}
    \sum_{i=-\infty}^\infty \left|\frac{\desCs{\mb}_{\sup}}{\eval - \desCs{\eval}_{\mu,i}}\right|^2
    <
    2\sum_{i=1}^\infty \left|\frac{\desCs{\mb}_{\sup}}{\frac{d}{4}i}\right|^2
    =
    M_\mu < \infty,
    && \forall \eval\not\in D
  \end{align*}
  and the theorem is proven with $M=M_\kappa + M_\mu<\infty$.
\end{proof}

\subsubsection{Modal feedback approximation}
\label{sec:hcc_ctrl_approx}
Let
\begin{align*}
\uu(t)=\BKc_\scc \scc(t), && \BKc_\scc = \adverseComp{\BKu}_\scc +\adverseComp{\Kb}_\scc
\end{align*}
be the designed controller~\eqref{eq:hccf_controller}, with
$\adverseComp{\BKu}_\scc$ unbounded and $\adverseComp{\Kb}_\scc$ bounded.
Remark that the decomposition
$\BKc_\scc = \adverseComp{\BKu}_\scc +\adverseComp{\Kb}_\scc$ is not unique, cf.~\cite{Woittennek2017ifac}, and
the exact realization $\adverseComp{\BKu}\s(t) = \adverseComp{\BKu}_\scc\Trafo_\s^\scc\s(t)$ may be difficult in original coordinates.
In many cases, as in the example discussed in Section~\ref{sec:example_hyperbolic},
one can state a
$\BKu=\adverseComp{\BKu}+\Delta\adverseComp{\BKu}$, with
$\Delta\adverseComp{\BKu}$ bounded, such that the feedback $\BKu\s(t)$ can be realized exactly.
This can be achieved for example via measurement or an appropriate adjustment of the plant.
Furthermore, for the approximation of the feedback 
\begin{align}
  \label{eq:control_law_scc}
\uu(t) = \BKu\s(t) + \Kb_\scc\scc(t), && \Kb_\scc=\adverseComp{\Kb}_\scc-\Delta\adverseComp{\BKu}\,\Trafo_\scc^\s,
\end{align}
neither $\Delta\adverseComp{\BKu}$
nor $\Trafo_\scc^\s$ need to be calculated explicitly.

For a simple representation of the feedback approximation scheme to be developed,
the control law~\eqref{eq:control_law_scc} is rewritten as
\begin{align}
  \label{eq:control_law_sfc}
  \uu(t) = \BKu\s(t) + \Kb_\sfc\sfc(t), && \Kb_\sfc=\Kb_\scc\,\Trafo_\sfc^\scc,
\end{align}
with the state $\sfc(t)$, introduced in~\eqref{eq:def_nu_state}.
The state $\sfc(t)$ represents a section of the
flat output $\fo(t+\cdot)$ and as such, the transformations between
$\sfc(t)$ and $\scc(t)$ are easy to compute, cf.\ Equation~\eqref{eq:conn_fc_cc}.

For the \ac{hc} with a control law of the form \eqref{eq:control_law_sfc}, the
efficient feedback approximation according to~\cite{Woittennek2017ifac} will
be described in the following.
As pointed out in Section~\ref{sec:approx_fb}, only the bounded part of the control law will be approximated:
\begin{align*}
  \Kb_\sfc\sfc(t) \approx \Kb_\sfc\Trafo_\s^\sfc\s^n(t) = \sum_{i=1}^{n}\ms_i^n(t)\underbrace{\Kb_\sfc\Trafo_\s^\sfc\svec_i^n}_{k_i^n}, && k_i^n\in\Compl,
\end{align*}
cf.~\eqref{eq:fa_ctrl_approx_scheme}.
It remains to compute the feedback gains
$\{k_i^n\}$ and, therefore, the transformations $\{\svec_{\sfc,i}^n=\Trafo_\s^\sfc\svec_i^n\}$
have to be known.
To avoid the explicit computation of $\Trafo_\s^\sfc$, which is the most difficult part of this design method,
one can choose the modal basis $\{\svec_i^n = \evec^\star_i\}$, constituted by the eigenvectors
$\{\evec^\star_i\}$ of the open loop system (if $\EA$ is Riesz-spectral),
the controller intermediate system\footnote{
According to Lemma~\ref{lem:xu_sallet_theorem11}, the system operator of the
controller intermediate system $\ctrlIs{\EA}$ is Riesz-spectral.
} $\star=\ctrlIsText$
or the desired system $\star=\desCsText$.
For fixed $\star$ (dynamics) and arbitrary $\bullet$ (coordinates) the solution
$\bullet(t)=$ can be expressed by a linear combination of the eigenvectors $\{\evec^\star_{\bullet,i}\}_1^\infty$:
\begin{align*}
\bullet(t) = \sum_{i=1}^\infty \ms_i^\star(t)\evec^\star_{\bullet,i}\in\Ss_\bullet,
\end{align*}
with the modal state
\begin{align*}
  \ms_i^\star(t) = \scal{\s(t)}{\evec_i^{\star*}} = \scalf{\bullet}{\evec_{\bullet,i}^{\star*}}{\Ss_\bullet}, && i\in\Nats.
\end{align*}
Since the modal state
does not depend on the choice of the coordinates, it doesn't matter
in which coordinates the feedback gains
\begin{align*}
k_i^\star=k_i^n|_{\svec_i^n = \evec^\star_i}=\Kb_\bullet\evec^\star_{\bullet,i}, && i=1,...,n
\end{align*}
will be determined. Therefore, the explicit knowledge of the transformation $\Trafo_\s^\sfc$ is not
required if the transformations $\evec^\star_{\sfc, i}=\Trafo_\s^\sfc\evec^\star_i$ of the elements $\evec^\star_i$ are known.
Observe that these elements correspond to particular solutions $\s(t)=\evec_i^\star e^{\eval_i^\star t}$
of the initial value problem
$\dot\s(t) = \EA^\star\s(t),\,\s(0)=\evec_i^\star$ evolving exponentially in time.
With the state $\sfc(t)$ -- see Definition~\eqref{eq:def_nu_state} --
the transformed eigenvector $\evec^\star_{\sfc, i}=\Trafo_\s^\sfc\evec^\star_i$ can be determined
by looking at the respective modal part $\s(t)=\evec_i^\star e^{\eval_i^\star t}$ separately:
\begin{align*}
\fo(t+\theta) = (\Phi\s)(t+\theta) = (\Phi\evec_i^\star) e^{\eval_i^\star \theta}e^{\eval_i^\star t}.
\end{align*}
They are finally given by
$\{\evec_{\sfc,i}^\star\in\Ss_\sfc\,|\,\evec_{\sfc,i}^\star(\theta)=(\Phi\evec_i^\star) e^{\eval_i^\star \theta}\}$.
In summary, the approximation scheme reads
\begin{align*}
  \uu(t) \approx \BKu\s(t) + \sum_{i=1}^{n}\ms_i^\star(t)k_i^\star,
\end{align*}
with
\begin{align*}
  \ms_i^\star(t)=\scal{\s(t)}{\evec_i^{\star*}}, && k_i^\star = \Kb_\sfc \evec_{\sfc,i}^\star.
\end{align*}

\subsubsection{Hyperbolic observer canonical form}
\label{sec:hcc_hocc}
For several hyperbolic systems~\eqref{eq:system} exists a state transformation
\begin{align*}
\soc = \Trafo_\s^\soc\, \s \in \Ss_\soc = \Compl^N\times \Lz(\thu, \tho;\Compl), && \Trafo_\s^\soc:\Ss\to\Ss_\soc
\end{align*}
such that the system appears in new coordinates $\soc$ in the \ac{hocf}.
The \ac{hocf} generalizes the concept of the observer canonical form known from
finite-dimensional systems. It can be derived for example from the
input-output relation/equation,
cf.~\cite{Woittennek2013nds}.
It is defined as the adjoint form of the \ac{hccf}.
\begin{defn}[\Acl{hocf}]
  \label{def:hcof}
The \ac{hocf} is a system representation with boundary observation
\begin{align*}
(\dot{\soc}(t), \yy(t)) = \SysOpBos[,\soc]\soc(t) = (\EA_\soc \soc(t),\EC_\soc \soc(t))
\end{align*}
such that, the adjoint operator $\SysOpBos[,\soc]^*$ defines a system in \ac{hccf}.
\end{defn}

Note that $\EA_\soc$ can be formally interpreted as
differential operator
\begin{align}
  \label{eq:hocf_ea}
\EA_\soc h &= (0,h_1, ..., h_{N-1}, -\diff{\theta}h_{N+1}) - \acc\EC_\soc h, && h\in D(\EA_\soc)
\end{align}
with $\acc\in D(\EA_\soc^*)'$ and the output operator
\begin{align}
\EC_\soc h &= h_{N+1}(\tho), && h\in D(\EA_\soc).
\end{align}
This can be interpreted as a chain of integrator connected to the input
of a transport system, while both are perturbed by the output of the transport system,
via the output injection $\acc$.

\subsubsection{Observer design}
\label{sec:hcc_obs_design}
In the coordinates $\soc$ the design of an observer
\begin{align*}
\dot{\hat{\soc}}(t) &= \EA_\soc\hat{\soc}(t) + \Lc_\soc\yt(t) \\
\yt(t)&=\yh(t)-\yy(t) = \EC_\soc\hat{\soc}(t) - \EC_\soc{\soc}(t)
\end{align*}
is particular simple, since the observer gain is always of the form
\begin{align}
  \label{eq:obs_gain_hocf}
\Lc_\soc = \acc - \desOs{\acc},
\end{align}
which compensates the system dynamics $\EA$, c.f.\ \eqref{eq:hocf_ea}, and assigns the desired
dynamics $\desOs{\EA}$ to the observer error system.
Since Theorem~\ref{thm:hccf_ad}, written in terms of $(\desCs{\EA}_\scc,\desCs{\EB}_\scc)$,
also applies in terms of $(\desOs[*]{\EA}_\soc,\desOs[*]{\EC}_\soc)$,
the construction of $\desOs{\acc}$ follows the same lines that applied to the
construction of $\desCs{\acc}$ in the \ac{hccf},
see Section~\ref{sec:hcc_ctrl_design}.

\subsubsection{Modal observer gain approximation}
\label{sec:hcc_obs_approx}
Let
$\Lc_\soc = \adverseComp{\Lu}_\soc + \adverseComp{\Lb}_\soc$ be
the designed observer gain~\eqref{eq:obs_gain_hocf}, with
$\adverseComp{\Lu}_\soc\in D(\EA^*_\soc)'$ and $\adverseComp{\Lb}_\soc\in\Ss_\soc$.
As described in Section~\ref{sec:approx_obs} for a similar decomposition of the observer gain, the observer should be rewritten in the form
\begin{align}
  \label{eq:obs_hat_adverse}
\dot{\hat{\soc}}(t) &= \obsIs{\adverseComp{\EA}}_\soc\hat{\soc}(t) - \adverseComp{\Lu}_\soc\yy(t) + \adverseComp{\Lb}_\soc\yt(t),
\end{align}
with $\obsIs{\adverseComp{\EA}}_\soc = \EA_\soc+\adverseComp\Lu_\soc\EC_\soc$, in order to derive
an approximation,
ensuring the convergence of the spectrum of the closed-loop system
in the sense of Lemma~\ref{lem:conv_obs}.

To avoid the explicit computation of the transformation $\Trafo_\s^\soc$ to the \ac{hocf}, the observer will be approximated in original coordinates
\begin{align}
  \label{eq:obs_hat_adverse_s}
\dot{\hat{\s}}(t) &= \obsIs{\adverseComp{\EA}}\hat{\s}(t) - \adverseComp{\Lu}\yy(t) + \adverseComp{\Lb}\yt(t).
\end{align}
However, since the unbounded part of the observer gain $\adverseComp{\Lu}=\Trafo_\soc^\s\adverseComp{\Lu}_\soc$ was designed in the \ac{hocf},
it is not necessarily simple to express in original coordinates.
Therefore, similarly to the feedback design, it is advantageous to
introduce another unbounded operator $\Lu=\adverseComp{\Lu} + \Delta\adverseComp{\Lu}$,
with $\Delta\adverseComp{\Lu}\in\Ss$,
such that ${\obsIs{\EA}} = \EA+\Lu\EC$ can be stated without the explicit knowledge of~$\Trafo_\soc^\s$.
This way, an observer approximation can be derived from
\begin{align*}
\dot{\hat{\s}}(t) &= {\obsIs{\EA}}\hat{\s}(t) - {\Lu}\yy(t) + {\Lb}\yt(t), && \Lb=\adverseComp{\Lb} - \Delta\adverseComp{\Lu}.
\end{align*}
However, for the approximation of the bounded operator $\Lb$ it is easier to switch back to
\ac{hocf} coordinates. In these coordinates, the observer-feedback gain
\begin{align*}
{\Lb}_\soc = \Trafo_\s^\soc \Lb = \obsIs{\acc} - \desOs{\acc} \in \Ss_\soc
\end{align*}
compensates the dynamics $\obsIs{\EA}$, determined by $\obsIs{\acc}$, of the
intermediate system and assigns the desired dynamics $\desOs{\EA}$, determined
by $\desOs{\acc}$.

While the controller approximation in Section~\ref{sec:hcc_ctrl_approx} was stated for arbitrary
modal basis elements $\{\evec^\star_{i}\}$, in the following, the observer gain approximation
will be derived for the eigenvectors $\{\obsIs{\evec}_{i}\}$ of the observer intermediate system only,
since the approximation scheme becomes particular simple in this case, see \eqref{eq:ubo_approx}.
The observer gain approximation
\begin{align}
  \label{eq:bounded_obs_gain_approx}
\Lb^n=\sum_{i=1}^n \obsIs{l}_i\,\obsIs{\evec}_i\,\in\,\Ss, && \obsIs{l}_i=l_i^n|_{\svec^n_i=\obsIs{\evec}_i}
\end{align}
is determined by
\begin{align*}
\obsIs{l}_i&=\scal{\Trafo_\soc^\s\Lb_\soc}{\obsIs[*]{\evec}_i}=\scalf{\Lb_\soc}{\Trafo_\soc^{\s*}\obsIs[*]{\evec}_{i}}{\Ss_\soc},
\end{align*}
where $\Trafo_\soc^{\s*}$, as the adjoint of a transformation to \ac{hocf}, is a transformation to \ac{hccf}.
As a consequence, the transform
$\obsIs[*]{\evec}_{\soc,i}=\Trafo_\soc^{\s*}\obsIs[*]{\evec}_{i}$ of
$\obsIs[*]{\evec}_{i}$ is not only an eigenvector of $\obsIs[*]{\EA}_\soc$ but also of $\obsIs{\EA}_\scc$,
where $\obsIs{\EA}_\scc$ describes the dynamics $\obsIs{\EA}$ of the observer intermediate
system in \ac{hccf} coordinates, cf.~\eqref{eq:hccf_bc_system}.
The structure of these eigenvectors
is always given by
\begin{align}
  \label{eq:evec_o_trafos}
  \obsIs[*]{\evec}_{\soc,i}=r_i(1,\cconj{\obsIs{\eval}}, ..., \cconj{\obsIs{\eval}}^{N-1}, \theta\mapsto e^{\cconj{\obsIs{\eval}}\,\theta}\,\cconj{\obsIs{\eval}}^N),
\end{align}
where $\cconj{\obsIs{\eval}}$ is the conjugate complex of $\obsIs{\eval}$.
It remains to compute the correct scaling $r_i$, such that $\obsIs[*]{\evec}_{\soc,i}=\Trafo_\soc^{\s*}\obsIs[*]{\evec}_{i}$ holds.
To this end, one can compute the flat output $\foa$
of the adjoint system
in terms of the original state $\foa=\Psi\s^*$
as well as in terms of the transformed state $\foa=\Psi_\soc\soc^*$
and adjust $r_i$
such that $\Psi_\soc\obsIs[*]{\evec_{\soc,i}}=\Psi\obsIs[*]{\evec_{i}}$.
\begin{rem}
  \label{rem:det_dual_flat_outs}
  For the calculation of the flat output $\foa=\Psi_\soc\,\soc^*$,
  the following general formula
  can be used: For $h=(h_1,...,h_N,h_{N+1}(\cdot))\in\Ss_\soc$,
\begin{align}
  \label{eq:hccf_flat_out}
  \Psi_\soc h = \left\{
    \begin{aligned}
    &h_1(0) = h(0)  && \text{for}\,\,N=0 \\[5pt]
    &\begin{aligned}
    &\sum_{i=0}^{N-1} \frac{(-\thu)^i}{i!}h_{i+1} &\\
    &+ \int_{\thu}^0 \frac{(-\theta)^{N-1}}{(N-1)!} h_{N+1}(\theta)\,d\theta
    \end{aligned} && \forall\, N \in \Nats\setminus\{0\}.
    \end{aligned}
  \right.
\end{align}
However, the corresponding operator $\Psi$ in the original coordinates has to be
determined by calculating the flat output $\foa=\Psi\s^*$ of each individual adjoint
system $\dot\s^*(t) = \obsIs[*]{\SysOpBos}(\s^*(t), \yy^*(t))$.
Of course, $\obsIs[*]{\SysOpBos}$ is the adjoint of the observer intermediate system operator $\obsIs{\SysOpBos}$
and not from the original system operator $\SysOpBos$.
\end{rem}

The last challenging part, the determination of $\obsIs{\acc}$, simplifies in
the context of this specific modal approximation to a point evaluation of $\obsIs[*]{\evec_{\soc,i,N+1}}$:
\begin{align}
\label{eq:ubo_approx}
\obsIs{l}_i&=\scalf{\obsIs{\acc} - \desOs{\acc}}{\obsIs[*]{\evec_{\soc,i}}}{\Ss_\soc} \\
&=-\cconj{\obsIs[*]{\evec_{\soc,i,N+1}}(\tho)} - \dualf{\desOs{\acc}}{\obsIs[*]{\evec_{\soc,i}}}{\obsIs[*]{\EA}_\soc},
\end{align}
cf.~\eqref{eq:hccf_bc_bc} and \cite{Riesmeier2021pamm}.

\subsection{Application}
\label{sec:example_hyperbolic}
To demonstrate the application of the theory and the proposed approximation
schemes, consider the simple\footnote{
This example was deliberately chosen, to focus on
the application of the presented approximation methods and the consequences of
the obtained results.
Furthermore, for this simple example, the state transformations, the state
feedback and the observer gain, whose explicit calculation will be
avoided in the following, can still be calculated relatively straightforward, cf.~\cite{Woittennek2012at}.
While this is a nice feature to verify the proposed approximation schemes for this
example, the approximation method can be applied in the same way to
systems, where the explicit computation of the transformations and gains is
more involved, see for example \cite{WoittennekWang2014ifac} and \cite{EcklebeEtal2017}.
}\footnote{
  Remark that, for the controller and observer design, also the backstepping method
  can be applied~\cite{KernEtal2018,GehringWoittennek2021}.
} hyperbolic system%
\begin{subequations}
\label{eq:ex_sys}
\begin{align}
\begin{split}
\label{eq:ex_dgl1}
\diff{t}\sw_1(z,t) &= \ca \diff{z}\sw_2(z,t)
\end{split} \\
\begin{split}
\label{eq:ex_dgl2}
\diff{t}\sw_2(z,t) &= \cb \diff{z}\sw_1(z,t)
\end{split} \\
\begin{split}
\label{eq:ex_dgl3}
\diff{t}\sw_3(t) &= \cc \sw_2(0,t)
\end{split}
\end{align}
with boundary conditions
\begin{align}
\label{eq:ex_bc}
\sw_3(t) = \sw_1(0,t), && \uu(t) = \sw_2(1,t),
\end{align}
\end{subequations}
input $\uu(t)$, and output $\yy(t) = \sw_1(1,t)$. This model can be used
to describe the linearized dynamics of an undamped pneumatic system \cite{Gehring2018mathmod}.
For this example, the design and approximation steps for the state feedback and the state
observer (with $\uu(t)=0$) will be shown separately.
Introducing the state
\[\s(t)=(\sw_1(\cdot,t),\sw_2(\cdot,t),\sw_3(t))\in\Ss=\Lz(0,1;\Compl^2)\times\Compl,\]
\eqref{eq:ex_sys} can be written as system with boundary control \eqref{eq:bc_system} with
\begin{align*}
D(\BA)  \! &=  \! \{(h_1, h_2, h_3) \!\in \!\Ss|\diff{z}h_1, \diff{z}h_2\!\in\!\Lz(0,1),h_3\!=\!h_1(0)\}\\
\BA h \! &=  \! 
(\ca\diff{z}h_2,\, \cb\diff{z}h_1,\, \cc h_2(0)), \qquad h\in D(\BA) \\
\BR h \! &=  \! h_2(1)
\end{align*}
or in state space representation \eqref{eq:system} with
\begin{align*}
D(\EA) &= \{(h_1, h_2, h_3) \in D(\BA)\,|\,h_2(1)=0\}\\
\EA\,h &= \BA\,h , \qquad\qquad\qquad h\in D(\EA) \\
\EB &= (\ca\,\Dirac{1}(\cdot),0,0).
\end{align*}
For the state space representation~\eqref{eq:bo_system} of the autonomous system with output it remains to specify the output operator:
\begin{align*}
\EC h &= h_1(1), && h\in D(\EA).
\end{align*}

\subsubsection{Controller design and approximation}
\label{sec:example_hyp_des_appr}
A flat output of the system is given by $\fo(t)=\Phi\s(t)=\sw_3(t)$.
The parameterization of the solutions by this flat output
\begin{subequations}
\label{eq:flat_param}
\begin{align}
\begin{split}
\label{eq:flat_param_w1}
\sw_1(z,t) &=\frac{1}{2}\big(\fo(t+\tau z) + \fo(t-\tau z)\big) \\
&\hphantom{=}+\frac{1}{2\cb\cc\tau}\big(\dot\fo(t+\tau z)-\dot\fo(t-\tau z)\big)
\end{split}\\
\begin{split}
\label{eq:flat_param_w2}
\sw_2(z,t) &=\frac{\cb\tau}{2}\big(\fo(t+\tau z) - \fo(t-\tau z)\big) \\
&\hphantom{=}+\frac{1}{2\cc}\big(\dot\fo(t+\tau z)+\dot\fo(t-\tau z)\big),
\end{split}
\end{align}
\end{subequations}
with the constant delay $\tau=v^{-1}$ and the velocity of propagation $v=\sqrt{\ca\cb}$,
can be derived -- for example -- via Laplace transform or the method of characteristics.
Evaluating \eqref{eq:flat_param_w2} at $z=1$ yields
\begin{align}
  \label{eq:example_hccf_bc}
  \uu(t)\! =\!\frac{\cb\tau}{2}\big(\fo(t+\tau) \!- \!\fo(t-\tau)\big) 
\!+\!\frac{1}{2\cc}\big(\dot\fo(t+\tau)\!+\!\dot\fo(t-\tau)\big).
\end{align}
Therefore, the integrator chain of the corresponding \ac{hccf}\footnote{
  From \eqref{eq:example_hccf_bc} follows that $a_m\neq 0$ and therewith $\EA_\scc$
  respectively $\EA$ is Riesz-spectral and a generator of a \CN-semigroup, cf. Section~\ref{sec:hcc_ctrl_design}.
}
has length $N=1$.
As already mentioned in Section~\ref{sec:hcc_ctrl_design} 
a stable delay differential equation of the form
\begin{align}
\label{eq:des_dyn}
\dot\fo(t+\tau)+\muc\,\dot\fo(t-\tau) + \kappac(\fo(t+\tau)+\muc\,\fo(t-\tau)) = 0,
\end{align}
with the design parameters $\kappac > 0$ and $\muc\in(-1,1)\setminus\{0\}$,
can be used for the desired closed-loop dynamics. 
The corresponding feedback
\begin{align}
\begin{split}
\label{eq:ex_ctrl_law}
\uu(t) = \frac{\cb\tau(\muc - 1)}{\muc + 1} \sw_1(1,t) + \frac{\cb\cc\tau - \kappac}{\cc(\muc + 1)}\fo(t+\tau) \\
+\frac{\muc(\cb\cc\tau+\kappac)}{\cc(\muc + 1)}\fo(t-\tau)
\end{split}
\end{align}
can be derived via linear combination of \eqref{eq:flat_param} (evaluated at $z=1$),
and \eqref{eq:des_dyn}, such that the
time derivatives of $\fo(t)$ will be eliminated.
This way, \eqref{eq:ex_ctrl_law} can be rewritten in the form
\begin{align*}
\uu(t) = \BKu\s(t) + \Kb_\sfc\sfc(t), && \s\in \Ss, &&
\sfc \in \Ss_\sfc=\Hn{1}(-\tau,\tau;\Compl),
\end{align*}
with the state $\sfc(t)$, according to \eqref{eq:def_nu_state},
the unbounded operator $\BKu$, defined by
\begin{align*}
  \BKu h = \frac{\cb\tau(\muc - 1)}{\muc + 1} h_1(1), && h=(h_1,h_2, h_3)\in D(\BA)
\end{align*}
and the bounded operator $\Kb_\sfc$, defined by
\begin{align*}
  \Kb_\sfc h &= \frac{\cb\cc\tau - \kappac}{\cc(\muc + 1)}h(\tau)
+\frac{\muc(\cb\cc\tau+\kappac)}{\cc(\muc + 1)}h(-\tau), && h\in\Ss_\sfc,
\end{align*}
cf. \eqref{eq:control_law_sfc}.
As described in Section~\ref{sec:hcc_ctrl_approx}, for the modal approximation scheme
\begin{align*}
  \uu(t) &= \BKu\s(t) + \Kb^n\s(t), &&
  \Kb^n = \sum_{i=1}^{n}\ms_i^\star(t)k_i^\star
\end{align*}
with weights
$\{\ms_i^\star(t)=\scalf{\s(t)}{\evec^{\star*}_{i}}{\Ss}\}$
and feedback gains
\begin{align*}
  \{k_i^\star = \Kb_\sfc\evec_{\sfc,i}^\star = (\Phi\evec_i^\star) \Kb_\sfc e^{\eval^\star_i \theta} = \evec_{i,3}^\star \Kb_\sfc e^{\eval^\star_i \theta}\},
\end{align*}
the transformation $\Trafo_\s^\sfc$ does not need to be determined explicitly.

\begin{figure}
\begin{center}
\includegraphics[width=\linewidth]{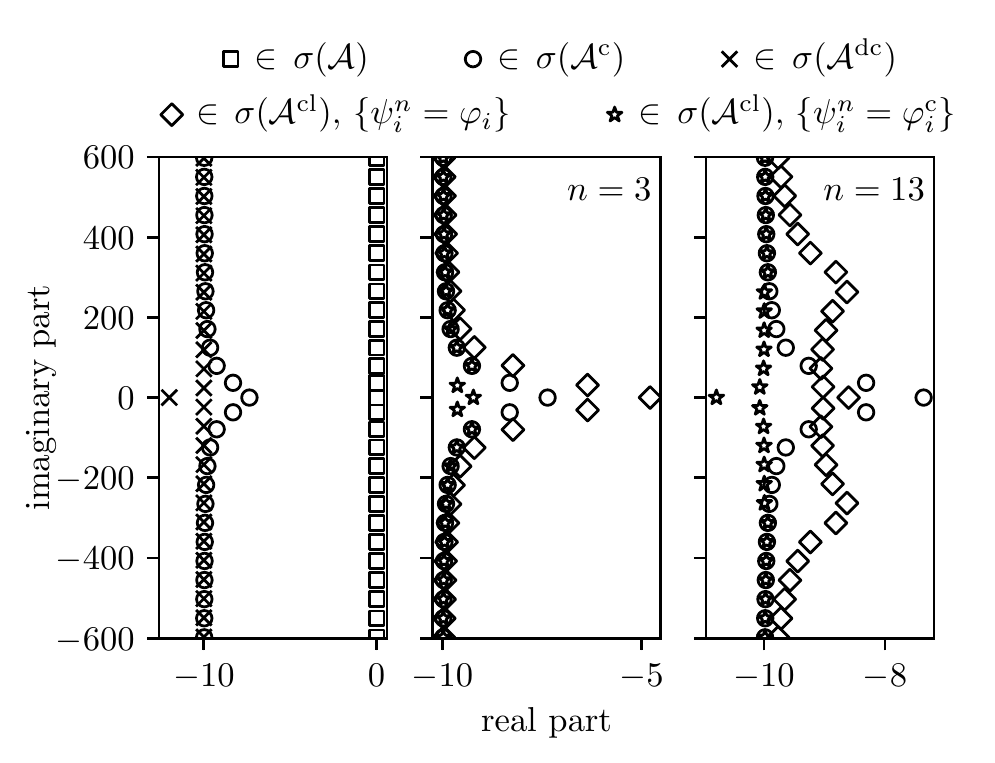}
\end{center}
\caption{Spectra of the different operators related to the state
feedback design and approximation, described in Section~\ref{sec:approx_fb},
for system~\eqref{eq:ex_sys}.}
\label{fig:fb}
\end{figure}

Figure~\ref{fig:fb} shows the spectra, related to the feedback design and
approximation, with the parameters
\begin{align}
\label{eq:ex_params}
\ca=11, && \cb=21, && \cc=31, &&
\muc=e^{-20 \tau}, && \kappac=12.
\end{align}
It can be seen that even a low approximation order of $n=3$ places the closed loop spectrum close to the desired one, provided 
the eigenvectors of the controller intermediate system will be used as approximation basis, i.e.,  $\{\svec_i^n\}_1^n=\{\ctrlIs{\evec}_i\}_1^n$.
In contrast, $n\ge 13$ is required to ensure a similar stability margin when using the eigenvectors of $\EA$ as approximation basis, i.e., $\{\svec_i^n\}_1^n = \{\evec_i\}_1^n$.
In particular, for the choice $\{\svec_i^n\}=\{\ctrlIs{\evec}_i\}$,
the spectrum of the closed-loop system is given by
$\sigma(\clS{\EA})=\sigma(\ctrlIs{\MA}_n+{\MB}_n{\MK}_n^\Transpose)\cup(\sigma(\ctrlIs{\EA})\setminus\{\ctrlIs{\eval}_i\}_1^n)$,
where the matrices are defined by
\begin{align*}
(\ctrlIs{\MA}_n)_{j,i}&=\scal{\ctrlIs{\EA}\ctrlIs{\evec}_i}{\ctrlIs[*]{\evec}_j}, && \ctrlIs{\MA}_n \in \Compl^{n\times n} \\
({\MB}_n)_i&=\ca\,\cconj{\ctrlIs[*]{\evec_{i,1}}(1)}, && \ctrlIs{\MB}_n \in \Compl^{n}  \\
({\MK}_n)_i&=\Kb\ctrlIs{\evec}_i, && \ctrlIs{\MK}_n \in \Compl^{n}.
\end{align*}
Hence, the closed loop spectrum $\sigma(\clS{\EA})$ can be determined by the
calculation of the spectrum of an $n\times n$ matrix, provided that
$\sigma(\ctrlIs{\EA})$ is known.

\subsubsection{Observer design and approximation}
As for the controller design, a desired stable delay differential equation
\begin{align}
\label{eq:des_dyn_obs}
\dot{\yt}(t+\tau)+\muo\,\dot{\yt}(t-\tau) + \kappao(\yt(t+\tau)+\muo\,\yt(t-\tau)) = 0,
\end{align}
for the observer output error $\yt(t)=\yh(t) - \yy(t)$ is prescribed for the observer dynamics,
with design parameters $\kappao > 0$ and $\muo\in(-1,1)\setminus\{0\}$.

According to Section~\ref{sec:approx_obs}, for the observer~\eqref{eq:obs_only_cl_sys_obs},
$\obsIs{\EA}$, $\Lu$ and $\Lb$ have to be determined.
As described in Section~\ref{sec:hcc_obs_approx}, the decomposition
$\Lc=\Lu+\Lb$ is not unique and therewith, $\obsIs{\EA}$ is not unique.
The decomposition $\Lc_\soc = \adverseComp{\Lu}_\soc + \adverseComp{\Lb}_\soc$,
which can directly derived from the observer gain~\eqref{eq:obs_gain_hocf},
can be used to state the observer~\eqref{eq:obs_hat_adverse}.
To avoid the computation of $\Trafo_\soc^\s$ and $\EA_\soc$,
the observer~\eqref{eq:obs_hat_adverse_s} in original coordinates, has to be determined.
But in original coordinates the corresponding decomposition
$\Lc=\adverseComp{\Lu} + \adverseComp{\Lb}$ is inconvenient,
compare Section~\ref{sec:hcc_obs_approx}.
Therefore, in the following, a convenient decomposition $\Lu={\Lu} + \Delta{\Lu}$
will be derived such that it is easy to express the observer intermediate dynamics in
original coordinates: $\obsIs{\EA}$.

Since the desired dynamics of the observer design~\eqref{eq:des_dyn_obs} differ only in parameters
from the desired dynamics of the controller design~\eqref{eq:des_dyn}, the results from
Section~\ref{sec:example_hyp_des_appr} can be used to derive $\Lu$.
From the controller design of the previous section, one knows that
the feedback $\uu(t)=\BKu\s(t)=\frac{\cb\tau(\muc - 1)}{\muc + 1}\BC\s(t)$,
for the boundary control system $\SysOpBcs(\s(t),\uu(t))$,
is an appropriate choice
to realize the unbounded part of the desired dynamics, resp.\ 
the neutral part of the delay differential
equation~\eqref{eq:des_dyn_obs}. Since this feedback gain is just the scaled output operator it can
also be used to realize the
neutral part of the delay differential equation~\eqref{eq:des_dyn_obs} in the
observer intermediate system.
Therewith, $\Lu$ is a scaled input operator of the controller intermediate system (with $\muc=\muo$)
\begin{align*}
  \Lu=\rho\obsIs{\EB}, && \rho=\frac{\cb\tau(\muo - 1)}{\muo + 1},
\end{align*}
where $\obsIs{\EB} = (\ca\,\Dirac{1}(\cdot),0,0)\in D(\obsIs{\EA})'$,
and
\begin{align*}
  D(\obsIs{\EA})=\{h_\s\in D(\BA) \,|\, \obsIs{\BR}h_\s=0 \}, && \obsIs{\BR}=\BR - \BKu|_{\muc=\muo}.
\end{align*}

It remains to determine
\begin{align*}
\obsIs{l}_i&=\scal{\obsIs{\acc} - \desOs{\acc}}{\obsIs[*]{\evec_{\soc,i}}} \\
&=-\cconj{\obsIs[*]{\evec_{\soc,i,2}}(\tau)}-\dualo[\soc]{\desOs{\acc}}{\obsIs[*]{\evec_{\soc,i}}},
\end{align*}
for the observer gain approximation~\eqref{eq:bounded_obs_gain_approx}.
To this end, the desired operator $\desOs{\acc}$
and the transformed adjoint eigenvectors $\{\obsIs[*]{\evec_{\soc,i}}= \Trafo_\s^\soc\obsIs[*]{\evec_{i}}\}_1^n$
have to be computed.
The former can be derived from \eqref{eq:des_dyn_obs} or from the general
formula \cite[Equation~11]{Woittennek2012at}:
\begin{align*}
\desOs{\acc}= \big(\kappao ( 1 + \muo), \kappao + \muo\,\Dirac{\tau}(\cdot)\big) \in D(\obsIs[*]{\EA}_\soc)',
\end{align*}
the latter, known from~\eqref{eq:evec_o_trafos}, are given by
\begin{align*}
\obsIs[*]{\evec_{\soc,i}}=r_i(1,\theta\mapsto e^{\cconj{\obsIs{\eval}_i}\,\theta}\,\cconj{\obsIs{\eval}_i}).
\end{align*}
Therein, $r_i$ has to be adjusted such that
$\Psi_\soc\obsIs[*]{\evec_{\soc,i}}=\Psi\obsIs[*]{\evec_{i}}$, cf. Section~\ref{sec:hcc_obs_approx}.
While, according to~\eqref{eq:hccf_flat_out},
\begin{align*}
&&\foa(t)= \Psi_\soc \xi(t) = \xi_{1}(t)+ \int_{\thu}^0 \xi_{2}(\theta,t)\,d\theta,&& \\
&&\xi(t)=(\xi_1(t),\xi_2(\cdot,t))\in\Ss_\soc,&&
\end{align*}
the corresponding map $\Psi$, which allows to compute the flat output of
the adjoint system in original coordinates, remains to be determined.
To this end, the adjoint system
$\dot\s^*(t) = \obsIs[*]{\SysOpBos}(\s^*(t), \yy^*(t))$
with state $\s^*(t)=(\sw_1^*(\cdot,t),\sw_2^*(\cdot,t),\sw_3^*(t))\in\Ss$
is considered in the form
\begin{align*}
\diff{t}\sw^*_1(z,t) &= -\cb \diff{z}\sw^*_2(z,t) \\
\diff{t}\sw^*_2(z,t) &= -\ca \diff{z}\sw^*_1(z,t) \\
\diff{t}\sw^*_3(t) &= -\cb \sw^*_2(0,t)
\end{align*}
with boundary conditions
\begin{align*}
\cc\sw^*_3(t) &= \ca\sw^*_1(0,t) \\
\yy^*(t) &= -\ca\rho\,\sw^*_1(1,t) - \cb\sw^*_2(1,t)
\end{align*}
and input $\yy^*(t)$ (cf. Remark~\ref{rem:det_dual_flat_outs}).
A flat output of the adjoint system is given by $\sw_3^*(t)$.
By looking at the flat parametrization of the input\footnote{
The analysis of the adjoint system can be traced back to the
already performed analysis of the original system, by applying the
state transformation $\check\sw_1^*(\cdot,t)=\ca\sw_1^*(\cdot,t)$, $\check\sw_2^*(\cdot,t)=-\cb\sw_2^*(\cdot,t)$ and $\check\sw_3^*(t)=\cc\sw_3^*(t)$.}
\begin{align}
  \label{eq:flat_param_dual_in}
  \begin{split}
  \yy^*(t) &= \frac{\cb\tau - \rho}{2\cb\tau}\dot\sw_3^*(t+\tau)+\frac{\cb\tau + \rho}{2\cb\tau}\dot\sw_3^*(t-\tau) \\
&\hphantom{=}+\cc\,\frac{\cb\tau-\rho}{2}\sw_3^*(t+\tau) - \cc\,\frac{\cb\tau+\rho}{2}\sw_3^*(t-\tau),
  \end{split}
\end{align}
it becomes clear that
\begin{align*}
\foa(t)=\Psi\s^*(t)=\frac{\cb\tau - \rho}{2\cb\tau}\sw_3^*(t).
\end{align*}
The scaling in the above definition of $\Psi$
emerges from a comparison of coefficients
of the first term in \eqref{eq:flat_param_dual_in} with the first term in \eqref{eq:hccf_bc_bc}.
Therefore, the transformed system, with the state according to \eqref{eq:conn_scc_scco} and \eqref{eq:conn_fc_cc}, appears in the \ac{hccf}.

Now $\{\obsIs[*]{\evec_{\soc,i}}\}_1^n$ can be adjusted via $\{r_i\}_1^n$,
such that
\begin{align*}
  r_i+ r_i\int_{\thu}^0 e^{\cconj{\obsIs{\eval}_i}\,\theta}\,\cconj{\obsIs{\eval}_i}\,d\theta=
  \frac{\cb\tau - \rho}{2\cb\tau}\obsIs[*]{\evec_{\soc,i,3}}
\end{align*}
and, therewith, also the approximation $\Lb^n$ can be computed.

The eigenvalue distribution of
$\sigma(\obsIs{\EA}+\Lb^n\obsIs{\EC})$
follows the same rules as the eigenvalue distribution of
$\sigma(\ctrlIs{\EA}+\ctrlIs{\EB}\Kb^n)$
from in Section~\ref{sec:example_hyp_des_appr}, see Figure~\ref{fig:fb}.

\section{Conclusion}
\label{sec:conclusion}

Late-lumping feedback and observer design for infinite-dimensional linear
systems with unbounded input and output operators are considered. The proposed
approximation schemes are inspired by \cite{Woittennek2017ifac,Riesmeier2018cdc}
and rely on a decomposition of the feedback and the output-injection gains into
a bounded and an unbounded part.  The approximation applies to the bounded part
only, while the unbounded part is assumed to allow for an exact realization.

Spectral convergence results for the closed-loop system operator, obtained with
the proposed approximation schemes, are provided.  The result relies on a set of
spectral assumptions provided in~\cite{XuSallet1996siam}. These assumptions
mainly concern the Riesz-spectral property of the system operator, the
eigenvalue distribution, and the modal expansion of the unbounded input
operator.

The problem under consideration is formulated such that the spectral assumptions
don't need to be checked for the given control system but rather for the desired
closed-loop system. As a consequence, the proposed design scheme can be applied
to systems, which do not completely satisfy the given assumptions, as long as
the desired closed-loop system does.

With the \ac{ac} and the \ac{hc} two important system classes have been studied
in more detail. For these systems, the assumptions required for the application
of the obtained results can be easily checked.  Moreover, for the \ac{hc}
particular target dynamics are proposed ensuring that these assumptions are
satisfied.  Furthermore, the controller and observer design is explained in
detail for the \ac{hc} and applied to an illustrative example.

Within the contribution, only observer gain approximation is considered, while
both observer approximation and finite-dimensional observer-based output feedback,
as done in~\cite{Deutscher2013IJC,GrueneMeurer2022} for an early-lumping design,
are not
treated. Although the underlying transformation-based techniques are well
suited for the design of observer-based output feedback, the convergence of the complete system dynamics
involving both feedback and observer approximation is left open for future
research.

\bibliographystyle{plain}

\end{document}